\documentclass[11pt,reqno]{amsart}
\usepackage{amscd}
\usepackage{amssymb}
\usepackage{amsthm}
\usepackage{amsfonts}
\usepackage{amsmath}
\usepackage{appendix}
\usepackage{arydshln}
\usepackage{boldline}
\usepackage{booktabs}
\usepackage{cancel}
\usepackage{comment}
\usepackage{calc}
\usepackage{color}
\usepackage{colortbl}
\usepackage{colonequals}
\usepackage{enumerate}
\usepackage{epstopdf}
\usepackage{fdsymbol}
\usepackage[margin=1.2in,letterpaper]{geometry}   
\usepackage{graphicx}
\usepackage{latexsym}
\usepackage{leftindex}
\usepackage{longtable}
\usepackage{makecell}
\usepackage{mathtools}
\usepackage{mhchem}
\usepackage{multirow}
\usepackage{nicematrix}
\usepackage{soul}
\usepackage{tensor}
\usepackage{tikz-cd}
\usepackage{xcolor}
\usepackage[all]{xy}
\usepackage{xypic}
\usepackage{ragged2e}

\usepackage[utf8]{inputenc}
\usepackage[colorlinks]{hyperref}
\hypersetup{
 colorlinks=true,
 linkcolor=blue,
 filecolor=magenta, 
 urlcolor=cyan,
}
\usepackage[nameinlink, capitalize]{cleveref}
\usepackage{colonequals}
\allowdisplaybreaks[1]

\newtheorem{theorem}{Theorem}[section]

\newtheorem*{introtheorem*}{Main Theorem}

\newtheorem{corollary}[theorem]{Corollary}
\newtheorem{lemma}[theorem]{Lemma}
\newtheorem{proposition}[theorem]{Proposition}

\theoremstyle{definition}
\newtheorem{definition}[theorem]{Definition}
\newtheorem*{defn*}{Definition}
\newtheorem{chunk}[theorem]{}
\newtheorem{ex}[theorem]{Example}

\newtheorem{remark}[theorem]{Remark}

\newtheorem{notation}[theorem]{Notation}
\newtheorem{setting}[theorem]{Setting}
\newtheorem{terminology}[theorem]{Terminology}
\theoremstyle{remark}

\numberwithin{equation}{theorem}
\newcommand{\BZ}{{\mathbb Z}}
\newcommand{\bsa}{{\boldsymbol a}}
\newcommand{\bsh}{{\boldsymbol h}}
\newcommand{\bsx}{{\boldsymbol x}}

\newcommand{\bsal}{{\boldsymbol \alpha}}
\newcommand{\bsbe}{{\boldsymbol \beta}}
\newcommand{\bsga}{{\boldsymbol \gamma}}
\newcommand{\bsla}{{\boldsymbol \lambda}}

\newcommand{\cA}{\mathcal{A}}
\renewcommand{\cD}{\mathcal{D}}

\newcommand{\wtP}{{\widetilde P}}
\newcommand{\wtR}{{\widetilde R}}
\newcommand{\wtA}{{\widetilde A}}
\newcommand{\wtI}{{\widetilde I}}
\newcommand{\wtJ}{{\widetilde J}}
\newcommand{\wtG}{{\widetilde G}}

\newcommand{\kk}{{\sf{k}}}
\newcommand{\m}{{\mathfrak m}}
\newcommand{\fa}{{\mathfrak a}}
\newcommand{\n}{{\mathfrak n}}
\newcommand{\p}{{\mathfrak p}}

\newcommand{\fh}{{\mathfrak h}}
\def\aa{\operatorname{\mathfrak a}}
\def\Hilb{\operatorname{Hilb}}
\def\Po{\operatorname{P}}
\def\Qo{\operatorname{Q}}
\def\Co{\operatorname{C}}

\def\ann{\operatorname{ann}}

\def\ch{\operatorname{char}}

\def\rank{\operatorname{rank}}
\def\reg{\operatorname{reg}}

\def\Tor{{\operatorname{Tor}}}
\def\Ext{{\operatorname{Ext}}}

\newcommand{\HH}{\operatorname{H}}
\def\Ker{\operatorname{Ker}}

\def\lin{\operatorname{lin}}
\def\Soc{\operatorname{Soc}}

\def\Span{{\operatorname{Span}}}
\def\ov{\overline}

\def\rate{\operatorname{rate}}

\newcommand{\dd}{\partial}
\newcommand{\Pu}[1]{P_{\,[#1]}}

\newcommand{\xra}{\xrightarrow}
\title[]{Homological properties of rings\\ defined by $n+1$ general quadrics in $n$ variables}


\author[Diethorn]{Rachel Diethorn}
\email{rdiethor@oberlin.edu}
\address{Department of Mathematics, Oberlin College, Oberlin, OH}
\author[G\"unt\"urk\"un]{Sema G\"unt\"urk\"un}
\email{s.gunturkun@essex.ac.uk}
\address{School of Mathematics, Statistics and Actuarial Science, University of Essex, UK}
\author[Hardesty]{Alexis Hardesty}
\email{ahardesty1@twu.edu}
\address{Division of Mathematics, Texas Woman's University, Denton, TX}
\author[Mete]{Pinar Mete}
\email{pinarm@balikesir.edu.tr}
\address{Department of Mathematics, Balikes\.{ı}r University, Turkey}
\author[\c{S}ega]{Liana \c{S}ega}
\email{segal@umkc.edu}
\address{Division of Computing, Analytics and Mathematics, University of Missouri-Kansas City, Kansas City, MO}
\author[Sobieska]{Aleksandra Sobieska}
\email{sobieskasnyd@marshall.edu}
\address{Department of Mathematics \& Physics, Marshall University, Huntington, WV}
\author[Veliche]{Oana Veliche}
\email{o.veliche@northeastern.edu}
\address{Department of Mathematics, Northeastern University, Boston, MA}

\begin{document}

%%%%%%%%%%%%%%%%%%%%%%%%%%%%%%%%%%%%%%%%%%%%%%%%%%%%%%%%%%%%%%%%%%%%%%%%%%%%%%%%%%%%%%%%%%%%%%%%%%%
\begin{abstract}
We study the almost complete intersection ring $R$ defined by $n+1$ general quadrics in a polynomial ring in $n$ variables over a field $\kk$ and a corresponding linked Gorenstein ring $A$. The overarching theme is that, while not Koszul (except for some small values of $n$), these rings have homological properties that extend those of Koszul rings.  We establish that finitely generated modules over these rings have rational Poincar\'e series and we give concrete formulas for the Poincar\'e series of $\kk$ over both $A$ and $R$. We also show that $A$ has minimal rate and its Yoneda algebra $\Ext_A(\kk,\kk)$ is generated by its elements of degrees $1$ and $2$.  While the graded Betti numbers of $R$ and $A$ over the polynomial ring are not known when $n$ is odd, our approach provides bounds and yields values for two of these Betti numbers, showing in particular that $R$ is level. 
\end{abstract}
\maketitle
%%%%%%%%%%%%%%%%%%%%%%%%%%%%%%%%%%%%%%%%%%%%%%%%%%%%%%%%%%%%%%%%%%%%%%%%%%%%%%%%%%%%%%%%%%%%%%%%%%%
\section{Introduction}

Polynomial rings over a field have the simplest homological behavior among graded algebras.  Namely, minimal free resolutions of finitely generated graded modules over these rings are finite, as given by Hilbert's syzygy theorem. There are two well-studied extensions of this kind of simple homological behavior. One extension is to the class of graded complete intersection rings, over which all finitely generated graded modules have a rational Poincar\'e series, and Betti numbers exhibit polynomial asymptotic growth; see \cite{Gulliksen74}. A second extension is to the class of Koszul algebras, over which any finitely generated graded module has finite regularity; see \cite{AP}. At the intersection of these two classes of rings lie quadratic complete intersections, which are arguably closest in homological behavior to polynomial rings. One obtains such rings when taking quotients of a polynomial ring in $n$ variables by $r\le n$ quadrics that are assumed 
{\it general},  meaning that they are parametrized by points in an appropriate nonempty Zariski open set. In this paper we consider quotients by  $n+1$ general quadrics, in which case the rings obtained are neither complete intersections nor Koszul (except for small values of $n$). 

An ideal defined by $n+1$ general forms in $n$ variables is an {\it almost complete intersection} (a.c.i.); namely its height is one less than the minimal number of generators. 
Results on monomial complete intersections, due to  Stanley \cite{Stanley80} and also Watanabe \cite{Wat} in characteristic zero and Cook \cite{cook} in positive characteristic, imply that such a.c.i.~ideals have the Strong Lefschetz Property under appropriate characteristic assumptions, and hence the Hilbert functions of the
quotient rings agree with the prediction of Fr\"oberg's conjecture given in \cite{Froberg-85}. Additionally, calculations of their Betti numbers over the polynomial ring have been provided in some cases, while conjectures have been proposed for other cases by Migliore and Mir\'o-Roig in \cite{MM}. 

Our contribution is an investigation of Poincar\'e series and Betti numbers over such rings when the a.c.i.~ideal is generated by general quadrics, and also over Gorenstein rings obtained from the a.c.i.~through linkage. In this setting, our results reveal homological behavior that is not too far from that of quadratic complete intersections. 

Before stating our  results, we recall, for comparison, that for a standard graded $\kk$-algebra $R$ the following statements are equivalent to the fact that $R$ is Koszul:
\begin{enumerate}
\item $\Po_\kk^R(t,u)=\Hilb_R(-tu)^{-1}$.
\item $\rate(R)=1$.
\item The Yoneda algebra $\Ext_R(\kk,\kk)$ is generated by its elements of degree $1$. 
\end{enumerate}
Here, $\Po_\kk^R(t,u)$ denotes the bigraded Poincar\'e series of $\kk$ over $R$, which is the generating function for the sequence of graded Betti numbers, $\Hilb_R(z)$ is the Hilbert series of $R$, and $\rate(R)$ (the Backelin rate of $R$) is an invariant that gives information about the shape of the Betti table of $\kk$ over $R$, see \cref{def:series-rate} for more details. In general, the rate of a singular ring $R$ is greater than or equal to $m-1$, where $m$ is the largest degree of a minimal generator of the defining ideal of $R$. Algebras with $\rate(R)=m-1$ are said to have {\it minimal rate}, and this property can be viewed as generalizing the Koszul property to non-quadratic algebras. 

The next theorem collects our main results. The terminology of {\it general quotient defined by $n+1$ quadrics} used here is made precise in \cref{def: generic}, where we first identify the parameter space for such rings and then we say that a property holds for a
general quotient if it holds for all parameter points in some nonempty Zariski open subset of that space.

\begin{introtheorem*}\label{introtheorem}
Let $n\ge 2$ be an integer and set $\ell=\left\lfloor\frac{n-2}{2}\right\rfloor$. 
Let $\kk$ be an infinite field of characteristic zero or greater than $n$ and consider the polynomial ring $Q=\kk[x_1, \dots, x_n]$. 
\begin{enumerate}[\rm (A)]
    \item A general quotient $R$ of $Q$ defined by $n+1$ quadrics $f_1, \dots, f_{n+1}$ has the following properties:
\begin{enumerate}[\quad$(1)$]
  \item The Poincar\'e series of $\kk$ over $R$ is given by 
  \vspace{0.1pt}
  
\[\begin{aligned}
  \Po_\kk^R(t,u)&={\displaystyle \frac{(-t)^{\ell}}{\Hilb_R(-tu)+(1-tu)^n\bigl((-t)^{\ell}-1\bigr)(1-t^2u^2)}}.
  \end{aligned}
  \]
  \vspace{0.1pt}
 
  Moreover, all finitely generated graded $R$-modules have rational Poincar\'e series.
  \item $\rate(R)=\frac{\ell}{2}+1$.
  \item $R$ is level, with socle in degree $n-\ell-1$ and type given by the Catalan number $C_{\ell+2}$.  
  \end{enumerate}  
  \medskip
  
\item Furthermore, when $n\ge 4$ the Gorenstein ring $A=Q/G$ with $G=(f_1, \dots, f_n)\colon (f_{n+1})$ associated to $R$ has the following properties:

 \begin{enumerate}[\quad$(1)$]
  \item The Poincar\'e series of $\kk$ over $A$ is given by
   \vspace{0.1pt}
   
  \[\begin{aligned}
  \Po_\kk^A(t,u)&={\displaystyle \frac{(-t)^{\ell-1}}{\Hilb_A(-tu)+(1-tu)^n\bigl((-t)^{\ell-1}-1\bigr)}}.
  \end{aligned}
  \]  
 \vspace{0.1pt}
 
  Moreover, all finitely generated graded $A$-modules have rational Poincar\'e series.
 \item $\rate(A)=\ell$.  In particular, $A$ has minimal rate.
 \item $\Ext_A(\kk,\kk)$ is generated as a $\kk$-algebra by its elements of degrees $1$ and $2$.
  \item The defining ideal $G$ of $A$ is generated in degrees $2$ and $\ell+1$.
\end{enumerate}
\end{enumerate}
\end{introtheorem*}

Properties (A.1) and (B.1) are proved in \Cref{t:generic} and \Cref{c:rationalPoincare}. Properties (A.2) and (B.2) are proved in \Cref{general-rate}. Property (B.3) is proved in \Cref{yoneda}.  Finally, properties (A.3) and (B.4) are known when $n$ is even; see \cref{rem: general - hilb}. When $n$ is odd, we prove them  in \Cref{c:RoverQ-known} and \Cref{c:AoverQ-known}, respectively; see also \cref{r:n=3} for the case $n=3$.

The characteristic assumption in Main Theorem ensures that the ring $R$ has the Strong Lefschetz Property, and thus the Hilbert series of both $R$ and $A$ are known. A key ingredient in our proofs arises from working over the complete intersection ring $P$ defined by $f_1, \dots, f_n$. In particular, we prove $\reg_P(A)=\ell$ (see \cref{generic reg}) and this allows us to prove that the induced maps $P\to R$ and $P\to A$ are both Golod homomorphisms. The details of proving that these maps are Golod are relegated to the Appendix, where we state a general criterion with wider applicability, see \cref{l:Massey-lemma} and \cref{p:Golod-I}.  Another essential ingredient comes from  \cite{proceedings}, where we study a special a.c.i.~ring, denoted here by $\widetilde R$ (see \cref{not:3}), which is used to show that the open sets involved in the definition of the concept of general are nonempty. 

Under the hypotheses of the Main Theorem, the Betti numbers of $R$ over $Q$ are known when $n$ is even, and conjectured when $n$ is odd; see \cite{MM}. Thus, it is somewhat surprising that we are able to give formulas for the Poincar\'e series of $\kk$ over $R$ independent of the parity of $n$. We make some contributions towards understanding the Betti numbers of $R$ over $Q$ when $n$ is odd as well; for example, the fact that $R$ is level fully describes the Betti numbers in the last homological degree, showing that they are equal to their conjectured values. We also obtain in \cref{p:overQ} bounds for all the Betti numbers of $R$ in terms of the Betti numbers of $\widetilde R$, which have been computed in \cite{proceedings}.

\section{Preliminaries}

In this section we record notation, terminology,  and some results from \cite{proceedings} that will be used in later sections. 

Let $S$ be a standard graded Noetherian $\kk$-algebra with $\kk$ a field and $M$ a finitely generated graded $S$-module with $n$th graded component $M_n$.  We denote by $\fh_M\colon\BZ \to \BZ_{\geq 0}$ the Hilbert function of $M$; that is,
\begin{align*}
\fh_M(n)&=\dim_{\kk}M_n\quad \text {for all $n\in \mathbb Z$}. 
\end{align*}
We also denote by $\Hilb_M(t)$ the Hilbert series of $M$; that is,
\begin{align*}
\Hilb_M(t)=\sum_{n= m}^{\infty}\fh_M(n)t^n, 
\end{align*}
where $m$ is the smallest integer with $\fh_M(m)\ne 0$.

For $i,j\in \mathbb Z$ we denote by $\beta_{i,j}^S(M)$ the $i$th graded Betti number of $M$ over $S$ in internal degree $j$, namely 
\[
\beta_{i,j}^S(M)=\rank_\kk(\Tor_i^S(M,\kk)_j)\,.
\]
We denote by $\Po_M^S(t,u)$ the (bigraded) Poincar\'e series of $M$ over $S$; that is, 
\begin{align*}
\Po_M^S(t,u)=\sum_{i,j\in\mathbb{Z}}\beta_{i,j}^S(M)t^iu^j.
\end{align*}
The classical Poincar\'e series of $M$ over $S$ is defined as $\Po_M^S(t)=\Po_M^S(t,1)$. 
The regularity of $M$ over $S$ is defined by 
\[
\reg_S(M)=\max\{j-i \,\vert\, \beta_{i,j}^S(M)\ne 0\}\,.
\]
The $S$-module $M$ is called  $a$-{\it linear} if $\beta_{i,j}^S(M)=0$ for all $i,j\geq 0$ with $j-i\not = a$.  A graded quotient ring $S/\aa$ is called $a$-{\it linear} if the defining ideal $\aa$ is $a$-linear as an $S$-module. The algebra $S$ is said to be {\it Koszul} if $\kk$ is $0$-linear. 

We recall a standard computation of Betti numbers in terms of Hilbert functions; see \cite[Lemma 2.1]{proceedings}.  For all integers $a,k$ with $k\ge 0$, we have: 
\begin{align}
\begin{split}
\label{Betti-recurrence}
\beta_{k,k+a}^S(M)=(-1)^k\fh_M(k+a)
-&\sum_{i=0}^{k-1}(-1)^{i+k}\fh_S(k-i)\beta_{i,i+a}^S(M)\\
-&\sum_{i=0}^{k+a} (-1)^{i+k}\sum_{\substack{j=i\\j\ne i+a}}^{k+a}
\fh_{S}(k+a-j)\beta_{i,j}^S(M)\,.
\end{split}
\end{align}

The following are some immediate consequences of \eqref{Betti-recurrence}.

\begin{chunk}  

\label{a-linear}
If $M$ is $a$-linear for some integer $a$, then 
\[
\Po_M^S(t,u)=\frac{\Hilb_M(-tu)}{(-t)^{a}\Hilb_S(-tu)}\,.
\]  
\end{chunk}

\begin{chunk}
\label{Betti-strand2}
Let $T$ be another standard graded Noetherian $\kk$-algebra, and $N$ a finitely generated graded $T$-module such that  $\fh_S=\fh_T$ and $\fh_M=\fh_N$. If there exists an integer $a\geq 0$ such that $\beta_{i,j}^S(M)=\beta_{i,j}^T(N)$ for all $i,j\ge 0$ with $j-i\ne a$, then $\beta_{i,j}^S(M)=\beta_{i,j}^T(N)$ for all $i,j\geq 0$. 
\end{chunk}

We now recall the notions of Golod homomorphisms, absolutely Koszul algebras, and Lefschetz properties.

\begin{chunk}\label{def:golod}
Let $\phi\colon S\to T$ be a surjective homomorphism of standard graded Noetherian $\kk$-algebras.  There is a coefficient-wise inequality on Poincar\'e series: 
\begin{align*}
    \Po_{\kk}^T(t,u)\preceq\frac{\Po_{\kk}^S(t,u)}{1-t(\Po_T^S(t,u)-1)}. 
\end{align*}
The map $\phi\colon S\to T$ is called a \textit{Golod homomorphism} if equality holds; see for example \cite[Proposition 3.3.2]{Avr}.  Additionally, see \cref{Golod-hom} in the Appendix for an equivalent definition. 
\end{chunk}

\begin{chunk}\label{def:abs_koszul}
Let $\m$ denote the maximal homogeneous ideal of $S$. If $F$ is a minimal graded free resolution of $M$ over $S$, then $\lin^S(F)$ denotes the associated graded complex of $F$ arising from the $\m$-adic filtration; see \cite{Herzog-Iyengar} for more details. The module $M$ is said to have {\it finite linearity defect} if $\HH_n(\lin^S(F))=0$ for $n\gg0$. The algebra $S$ is said to be {\it absolutely Koszul}  if every finitely generated graded $S$-module $M$ has finite linearity defect. It is known that any absolutely Koszul algebra $S$ is a Koszul algebra. 
Thus, if $S$ is absolutely Koszul then $\Po_\kk^S(t,u)={\Hilb_S(-tu)}^{-1}$ and any finitely generated graded $S$-module $M$ has a rational Poincar\'e series, with denominator  $\Hilb_S(-tu)$. 
\end{chunk}

\begin{chunk}{\cite[Proposition 5.8, Theorem 5.9]{Herzog-Iyengar}}
\label{prop:abs-koszul}
If  $\varphi\colon P\to S$ is a surjective homomorphism of standard graded $\kk$-algebras, $P$ is a quadratic complete intersection (and thus Koszul) and $\Ker(\varphi)$ has a $2$-linear resolution over $P$, then $S$ is absolutely Koszul. 
\end{chunk}

\begin{chunk} Let $f$ be a homogeneous element of $S$ of degree $d\geq 1$. We say that $f$ is a {\it maximal rank element} of $S$ if for each $i\ge 0$ the map  $S_{i}\xrightarrow{}S_{i+d}$  given by multiplication by $f$ has maximal  rank.  If $d=1$ (that is, if $f$ is a linear form) and $f$ is a maximal rank element, then we say that $f$ is a {\it Weak Lefschetz element}.  If there exists a linear element $f$ in $S$  such that for each $j\geq 1$  the power  $f^j$ is a maximal rank element of $S$, then we say that $f$ is a {\it Strong Lefschetz element} of $S$. 

We say that the ring $S$ has the {\it Weak} ({\it respectively Strong}) {\it Lefschetz Property}, or WLP (respectively SLP), if there exists a Weak (respectively Strong) Lefschetz element of $S$. Notice that if $S$ has the SLP with Strong Lefschetz element $f$, then it has the WLP with Weak Lefschetz element $f$, and $f^j$ is a maximal rank element for each $j\ge 1$.

Under appropriate characteristic assumptions, it is known that a complete intersection defined by general quadrics has the SLP, and hence they have maximal rank elements in any degree; more details are recorded in \cref{general}. 
\end{chunk}
\smallskip

We consider some specific standard graded rings and adopt the following notation and hypotheses for the remainder of the section. 

\begin{setting}
\label{not:1}
Let $\kk$ be a field, $Q=\kk[x_1, \dots, x_n]$ with $n\geq 2$, and $\{f_1,\dots, f_{n+1}\}$ a set of homogeneous elements in $Q$ of degree $2$. Define the following ideals of $Q$:
\begin{equation*}
J \coloneq (f_1, \dots, f_n),\quad I\coloneq (f_1, \dots, f_{n+1}),\quad \text{and}\quad G \coloneq  J:I. 
\end{equation*}
We further assume that $f_1,\ldots, f_n$ is a regular sequence and $f_{n+1}\notin J$, so that $J$ is a complete intersection ideal, $I$ is an almost complete intersection ideal, and $G$ is a Gorenstein ideal by \cite[Remark 2.7]{HU87}. We also define the following quotient rings of $Q$:
\begin{equation*}
P \coloneq Q/J,\quad R\coloneq  Q/I,\quad\text{and}\quad A\coloneq Q/G. 
\end{equation*}
Notice that $\ann_P(f_{n+1})=G/J$ and $A\cong P/\ann_P(f_{n+1})$. In this setting, $P$ is a complete intersection ring, $R$ is an almost complete intersection ring, and   $A$ is a Gorenstein ring.
We also set 
\[
\ell\coloneq \left\lfloor\frac{n-2}{2}\right\rfloor.  
\]
Observe that $n=2\ell+2$ when $n$ is even and $n=2\ell+3$ when $n$ is odd. Finally, for all integers $0\leq u\leq n$ we set
\[
\Pu{u} \coloneq Q/(f_1,\dots,f_{u}).
\]
with the convention that $\Pu{0}=Q$.  Notice also that $\Pu{n}=P$.
\end{setting}

\begin{remark}
\label{APR} 
There exists a short exact sequence of graded $P$-modules:
\begin{equation}
\label{ses APR}
0\to A(-2)\to P\to R\to 0,
\end{equation}
see for example \cite{Eis}, where the map $P\to R$ is the canonical projection and the map $A(-2)\to P$ is multiplication by $f_{n+1}$.  As a direct consequence of this exact sequence we have the following equalities: 
\begin{align}
     \label{e:reg}
     \reg_P(R)&=\reg_P(A)+1,\\
     \label{hilb APR}
     \Hilb_R(t)&=\Hilb_P(t)-t^2\Hilb_A(t), 
\end{align}
where the first equality follows from the isomorphism: 
\begin{equation*}
\Tor_i^P(R,\kk)_j\cong \Tor_{i-1}^P(A,\kk)_{j-2}\quad \text{for all}\  i\geq 1\text{ and } j\geq 0.
\end{equation*}
\end{remark} 

Next, we record known computations (see \cite{MM}) of the Hilbert functions of the rings $A$ and $R$ and, when $n$ is even, of their Betti numbers.  For consistency of notation, we refer to the analogous statements given in \cite{proceedings}.

\begin{chunk}\cite[Proposition 3.5]{proceedings}
\label{Hilb R and A}
If $f_{n+1}$ is a quadratic maximal rank element of $P$, then the following assertions hold.
\begin{enumerate}[\quad$(1)$]
\item The Hilbert function of $R$ is given by  
\begin{equation}
\label{hilbR}
\fh_{R}(i)= \max {\left\{ \binom{n}{i}-\binom{n}{i-2},0\right\}} \quad\text{for all}\quad i\geq 0.
\end{equation}
In particular, the socle of  $R$  has  maximum degree  $n-\ell-1$.

\item The Hilbert function of $A$ is given by
\begin{equation}
\label{hilbertA}
\fh_{A}(i)=\min\left\{ \binom{n}{i},\binom{n}{i+2}\right\} \quad\text{for all}\quad i\geq 0.
\end{equation}
In particular, we have:
\begin{enumerate}[\quad$(a)$]
    \item The socle degree of  $A$ is $n-2$.
    \item The $P$-ideal $G/J$ is generated in degrees at least $\ell+1$. 
\end{enumerate}
\item The following equalities hold: 
\[
\dim_\kk(\Soc R)_{n-\ell-1}= {\binom{n}{\ell+1}-\binom{n}{\ell+3}} = C_{\ell+2}=\dim_\kk(G/J)_{\ell+1}.
\]
where $C_{\ell+2}$ is the $(\ell+2)$-th Catalan number. 
\end{enumerate}
\end{chunk}

\begin{chunk} 
\label{not: rho}  \cite[Notation 3.9, Proposition 3.10]{proceedings}  For $n\geq 1$, define a sequence $\{\rho_k(n)\}_{k\ge 0}$ by setting $\rho_0(n)=0$ and for $k\geq 1$ using  the following recursive relation: 
\begin{align*}
\rho_k(n)=&\sum_{i=0}^{k-1}(-1)^{i+k+1}{\binom{n+k-i-1}{n-1}}\rho_i(n)+\sum_{i=0}^{\ell} (-1)^{i+k+1}  \binom{n+k+\ell-2i}{n-1}\binom{n+1}{i}.
\end{align*}

For example, when $n=6$, the values of $\rho_k(n)$ for $0\le k\le 6$ are:
\begin{equation}\label{e:small rho}
    0,0, 14,105,132,70,14.
\end{equation}

If $n$ is even and $f_{n+1}$ is a quadratic maximal rank element of $P$, then
\begin{align}\label{Betti RQ} 
\beta_{i,j}^Q(R)= \begin{cases} 
    \binom{n+1}{i},&\text{ if } j=2i\,  \text{ and }\,  0\leq i \leq\ell\\
       \rho_i(n), &\text{ if } j= i+\ell+1 \text{ and } i\ge 0\\
       0,&\text{ otherwise}\,.
    \end{cases}
\end{align}
where $\rho_1(n)=0$ when  $n\ge 4$. 
\end{chunk}

\begin{chunk}\label{not: gamma} \cite[Notation 3.12, Proposition 3.13]{proceedings}
For $n\geq 1$, define a sequence $\{\gamma_k(n)\}_{k\geq 0}$ by setting 
\[
\gamma_0(n)=\begin{cases}
0, &\text{if $n>2$}\\
1, &\text{if $n=2$}\,, 
\end{cases} 
\quad 
\gamma_k(n) = \gamma_{n-k}(n) \text{ for } 1\le k\le n,
\] 

and for $1\le k \le \ell+1$ using the following recursive relation:
\begin{align*}
\gamma_k(n) =& 
(-1)^k{ \binom{n}{k+\ell+2}} - \sum_{i=0}^{k-1}(-1)^{i+k}{ \binom{n-1+k-i}{n-1}}\gamma_i(n)\\
&- \sum_{i=0}^{\ell-1}(-1)^{i+k}{ \binom{n-1+\ell+k-2i}{n-1}\binom{n}{i}}.
\end{align*}

For example, for $n=6$, the values of $\gamma_k(n)$ for $0\leq k\leq 4=\ell+1$ are $0, 14, 85, 132$. If $n$ is even and $f_{n+1}$ is a quadratic maximal rank element of $P$, then
\begin{align}\label{Betti AQ}
\beta_{i,j}^Q(A)=\begin{cases} 
\binom{n}{i},&\text{ if } j = 2i   \text{ and } 0\leq  i \leq \ell-1 \text{ or  if }\ j = 2i-2   \text{ and } \ell+3\leq i \leq n\\
\gamma_i(n),&\text{ if}\ j=i+\ell \text{ and }\,  0\leq i\leq n\\
0,&\text{ otherwise}\,.
\end{cases}
\end{align}

where  $\gamma_1(n)=\gamma_{n-1}(n)=
C_{\ell+2}$. 
\end{chunk}

\section{Regularity over the complete intersection ring $P$}

In this section we work with \cref{not:1} and we state some results about the regularity of the almost complete intersection ring $R$ and the Gorenstein ring $A$ over the complete intersection ring $P$. In particular, we show that the resolutions of $R$ and $A$ over $P$ have the shapes below, and the last strand in both diagrams vanishes when $n$ is even. 
\[
\begin{array}[t]{r|cccccc}
\beta_{i,j}^P(R)&0&1&2&3&4&\cdots\\
\hline
0&1&-&-&-&-&\cdots\\
1&-&\text{1}&-&-&-&\cdots\\
2&-&-&-&-&-&\cdots\\
\text{\vdots}&\text{\vdots}&\text{\vdots}&\text{\vdots}&\text{\vdots}&\text{\vdots}&\text{\vdots}\\
\ell+1&-&-&*&*&*&\cdots\\ 
\ell+2&-&-&*&*&*&\cdots 
\end{array}
\hspace{2cm}
\begin{array}[t]{r|cccccc}   \beta_{i,j}^P(A)&0&1&2&3&4&\cdots \\
\hline
0&1&-&-&-&-&\cdots\\
1&-&-&-&-&-&\cdots\\
\text{\vdots}&\text{\vdots}&\text{\vdots}&\text{\vdots}&\text{\vdots}&\text{\vdots}&\text{\vdots}\\
\ell&-&*&*&*&*&\cdots\\ 
\ell+1&-&*&*&*&*&\cdots 
\end{array}
\]

Towards this goal, we begin by stating two useful ingredients and establishing some consequences, including a computation of the regularity of $R$ over the intermediate complete intersections $\Pu{u}$ for $0\leq u\leq n$.

\begin{chunk} \cite[Proposition 3.8~(1,2)]{proceedings}
\label{l: A and R over Pu}  If $f_{n+1}$ is a quadratic maximal rank element of $P$, then for each $0\leq u\leq n$  the following equalities hold:

\begin{enumerate}[\quad $(1)$]
\item $\beta_{i,j}^{\Pu{u}}(A)= 
\begin{cases} 
\binom{n-u}{i},&\text{ if } j=2i\,  \text{ and }\,  0\leq j-i <\ell\\       
\quad 0,&\text{ if $j\ne 2i$ and $0\leq j-i<\ell$}\,;
\end{cases}$
\item 
$\beta_{i,j}^{\Pu{u}}(R)=
\begin{cases} 
\binom{n-u+1}{i},&\text{ if } j=2i\,  \text{ and }\,  0\leq j-i \le \ell\\       
\quad 0,&\text{ if $j\ne 2i$ and  $0\leq j-i\le \ell$}\,;
\end{cases}$ 
\end{enumerate}
\end{chunk}

\begin{chunk} \cite[Proposition 3(2)]{Koszul-survey}
    \label{Ks}
Let $S$ be a standard graded $\kk$-algebra, $T$ a quotient ring of $S$ by a homogeneous ideal, and $M$ a finitely generated graded $T$-module. If $\reg_S(T)\le 1$, then 
$\reg_T(M)\le \reg_S(M)$. 
\end{chunk}

\begin{proposition} 
\label{reg}
Adopt the notation and hypotheses in \cref{not:1}. Then 
\begin{enumerate}[\quad$(1)$]
\item $\reg_{P}(R)=\reg_{\Pu{n}}(R)\le  \reg_{\Pu{n-1}}(R)\le \cdots \le \reg_{\Pu{0}}(R)=\reg_Q(R)$. \\
\end{enumerate}

Moreover, if $f_{n+1}$ is a quadratic maximal rank element of $P$, then:
\begin{enumerate}[\quad$(1)$]
\setcounter{enumi}{1}
\item For all $0\le u\le n$ the following inequalities hold:
    \[\ell+1\le \reg_{\Pu{u}}(R)\le \ell+2,\]    with $\reg_{\Pu{u}}(R)=\ell+1$ when $n$ is even. 
    \item The following inequalities hold: 
    \[\ell\le \reg_P(A)\le \ell+1,\]
    with $\reg_P(A)=\ell$ when $n$ is even. 
\end{enumerate}
\end{proposition}

\begin{proof}
$(1)$: Observe that for $0\le u \le n-1$, we have $\reg_{\Pu{u}}(\Pu{u+1})= 1$, as $f_{u+1}$ is a quadratic regular element on $\Pu{u}$ and $\Pu{u+1}=\Pu{u}/(f_{u+1})$. Therefore by successively applying \cref{Ks} with $T=\Pu{u+1}$ and $S=\Pu{u}$ for each $0\le u \le n-1$, we get the desired inequalities.

Assume now that $f_{n+1}$ is a quadratic maximal rank element of $P$.

$(2)$:  For all $0\le u\le n$ we have 
\[
\ell+1\le \reg_P(R)\le \reg_{\Pu{u}}(R)\le \reg_Q(R)=\begin{cases}\ell+2&\text{ if $n$ is odd}\\
\ell+1 &\text{ if $n$ is even.}
\end{cases}
\]
The first inequality follows from \cref{Hilb R and A}(2b), since $G/J$ is the second syzygy in a minimal free resolution of $R$ over $P$. 
The second and the third follow from part (1). The last equality follows from the fact that the socle of $R$ has maximum degree $n-\ell-1$ by \cref{Hilb R and A}(1) and the observation that $n-\ell-1$ equals  $\ell+2$ if $n$ is odd and $\ell+1$ if  $n$ is even.

$(3)$: The statement follows from (2) with $u=n$, in view of \eqref{e:reg}. 
\end{proof}

Now we are ready to prove that the Betti tables of $R$ and $A$ over $P$ have the desired shape. 

\begin{proposition} 
\label{p:Betti-over-P}
Adopt the notation and hypotheses in \cref{not:1}. If $f_{n+1}$ is a quadratic maximal rank element of $P$, then the following assertions hold. 
\begin{enumerate}[\quad $(1)$]
\item The Betti numbers of $R$ over $P$ satisfy the equalities 
\begin{align*}
\beta_{1,2}^P(R)&=1\quad \text{and}\quad \beta_{1,j}^P(R)=0\quad\text{for all } j\ne 2, \text{ and } \\
\beta_{i,j}^P(R)&=0\quad\text{for all}\quad j-i\notin \{\ell+1,\ell+2\}\quad\text{and}\quad i>1.
\end{align*}  
\item The Betti numbers of $A$ over $P$ satisfy the equalities
\[
\beta_{i,j}^P(A)=0\quad\text{for all}\quad j-i\notin \{\ell,\ell+1\}\quad\text{and}\quad i>0.
\]  
\item  If $\reg_{\Pu{1}}(R)=\ell+1$ (e.g. when $n$ is even), then 
\begin{enumerate}[\quad\rm(a)]
\item $\reg_P(R)=\ell+1$ and $\beta_{i,j}^P(R)=0$ for all $j-i\ne \ell+1$ and $i>1$.  
\item $\reg_P(A)=\ell$ and $\beta_{i,j}^P(A)=0$ for all $j-i\ne \ell$ and $i>0$.  
\item $\reg_{\Pu{1}}(A)=\ell$ and $\beta_{i,j}^{\Pu{1}}(A)=\binom{n-1}{i}$ when $j=2i-2$ and $j-i>\ell$.
\item The $P$-module $L=\ann_P(f_{n+1})$ has an $(\ell+1)$-linear resolution over $P$.
\end{enumerate}
\end{enumerate}
\end{proposition}

\begin{proof}
(1): For $i=1$ this follows by the definition of $R$ and $P$ and for $i>1$ from \cref{reg}(2). 

(2): This follows directly from \cref{l: A and R over Pu}(1) and \cref{reg}(3). 

(3a): By \cref{reg}(1), we have $\reg_P(R)\le \reg_{\Pu{1}}(R)=\ell+1$, and hence $\reg_P(R)=\ell+1$ follows from \cref{reg}(2). The statement about Betti numbers follows now from part (1).  

(3b): The statement about regularity follows from \eqref{e:reg} in \cref{APR} and thus the statement about Betti numbers follows from (2).

(3c): The short exact sequence \eqref{ses APR} in
 \cref{APR} induces the long exact sequence:
\begin{equation*}
          \Tor_{i+1}^{\Pu{1}}(R,\kk)_{j+2}\to \Tor_{i}^{\Pu{1}}(A,\kk)_{j}\to 
          \Tor_{i}^{\Pu{1}}(P,\kk)_{j+2}\to 
          \Tor_{i}^{\Pu{1}}(R,\kk)_{j+2}.
\end{equation*}
If $j-i>\ell$, then 
\[
(j+2)-i>(j+2)-(i+1)=(j-i)+1> \ell+1\,.
\]
Using the hypothesis that $\reg_{\Pu{1}}(R)=\ell+1$, we have $\Tor_{i+1}^{\Pu{1}}(R,\kk)_{j+2}=0=\Tor_{i}^{\Pu{1}}(R,\kk)_{j+2}$. Hence we have  
\[
\beta_{i,j}^{\Pu{1}}(A)=\beta_{i,j+2}^{\Pu{1}}(P)=
\begin{cases}\binom{n-1}{i} &\text {if } j=2i-2 \text{ and } j-i>\ell\\
0 &\text{if } j\ne 2i-2 \text{ and } j-i>\ell,
\end{cases}
\]
where the second equality follows since $P$ is a complete intersection over $\Pu{1}$ that is  generated by  $n-1$ quadrics.

(3d): This follows from the observation that $L=G/J$ is the first syzygy in a minimal free resolution of $A\cong P/(G/J)$ over $P$.
\end{proof}

Next, we show in \cref{p:reg ex} that there exist situations in which $\reg_{\Pu{1}}(R)=\ell+1$ when $n$ is odd. To do so, we will use the rings studied in \cite{proceedings}. Since these rings will play an important role in this paper, we introduce notation below. 

\begin{notation}
\label{not:3}
Define $Q=\kk[x_1, \dots, x_n]$, where $\kk$ denotes a field of characteristic zero or positive characteristic greater than $n$,  and the ideals:
\begin{equation*}
\wtJ \coloneq (x^2_1, \dots, x^2_n),\quad \wtI\coloneq (x^2_1, \dots,x_n^2, (x_1+\dots+x_{n})^2),\quad \text{and}\quad \wtG \coloneq  \wtJ : \wtI.
\end{equation*}
We also consider the quotient rings of $Q$:
\begin{equation*}
\wtP \coloneq Q/\wtJ,\quad \wtR\coloneq  Q/\wtI,\quad \wtA\coloneq Q/\wtG, \quad \text{and}\quad 
\wtP_{\,[1]}=Q/(x_1^2).
\end{equation*}
Notice that $\wtA \cong \wtP/\ann_{\wtP}((x_1+\dots+x_n)^2)$.
We also consider the standard graded rings that correspond to the ones above, but with $n-1$ variables instead of $n$ variables:
\begin{align*}
 \ov Q&\coloneq \kk[x_1\dots,x_{n-1}]&
 \ov R&\coloneq \ov Q/(x_1^2,\dots,x_{n-1}^2,(x_1+\dots+x_{n-1})^2)\\
 \ov P&\coloneq \ov Q/(x_1^2,\dots,x_{n-1}^2)&
 \ov A&\coloneq \ov Q/\big((x_1^2, \dots, x_{n-1}^2) : (x_1+\dots +x_{n-1})^2\big).
 \end{align*}
\end{notation}

\begin{remark}
\label{ex: wlp}

It is known that $x_1+\dots+x_n$ is a Strong Lefschetz element of $\wtP$, and hence $(x_1+\dots+x_n)^2$ is a maximal rank element of $\wtP$, when $\ch\kk=0$ by \cite[Proposition 2.2]{MMN11} or when $\ch \kk> n$ by \cite[Theorem 3.6(ii)]{cook}. In particular, the rings $\wtP$, $\wtR$, and $\wtA$ satisfy \cref{not:1}. Therefore, any result requiring only \cref{not:1} and $f_{n+1}$ is a maximal rank element can be applied for the rings  $\wtP$, $\wtR$, and $\wtA$.
\end{remark}

The main ingredient of \cref{p:reg ex} is a result from \cite{proceedings} which allows us to extend results known in the case when $n$ is even to the case when $n$ is odd. 

\begin{chunk}
\label{Betti reduction}
Set $T=Q/(x_n^2)$. Then, by symmetry, we have 
\begin{equation}
\label{e:symm}
\beta_{i,j}^{T}(\wtR)=\beta_{i,j}^{\wtP_{\,[1]}}(\wtR) \quad\text{and}\quad \beta_{i,j}^{T}(\wtA)=\beta_{i,j}^{\wtP_{\,[1]}}(\wtA)\quad\text{for all $i,j\ge 0$.}
\end{equation}
Therefore, with the  hypotheses in \cref{not:3}, we can restate \cite[Proposition 4.8]{proceedings} as follows.
If  $n\ge 3$ is odd, then the following equalities hold for all $i,j\ge 0$:
\begin{enumerate}[\quad $(1)$]
\item  $\beta_{i,j}^{\wtP_{\,[1]}}(\wtR)=\beta_{i,j}^{\ov Q}(\ov R)$;
\item  $\beta_{i,j}^{\wtP_{\,[1]}}(\wtA)=\beta_{i,j}^{\ov Q}(\ov A).$
\end{enumerate}
\end{chunk}

\begin{proposition}\label{p:reg ex}
Adopt the notation and hypotheses in \cref{not:3}. Then 
\[
\reg_{\wtP_{\,[1]}}(\wtR)=\ell+1. 
\]
In particular, the conclusions of \cref{p:Betti-over-P}(3) hold for the rings $\wtP$, $\wtR$, and $\wtA$.
\end{proposition}
\begin{proof}
When $n$ is even, the statements follow from \cref{reg}(2), in view of \cref{ex: wlp}. Assume now $n$ is odd. 
In this case, $\ell$ is also equal to $\left\lfloor \frac{(n-1)-2}{2}\right\rfloor$. Applying \cref{reg}(2) to the ring $\overline R$, we see that $\reg_{\,\ov Q}(\ov R)=\ell+1$, since $n-1$ is even. By \cref{Betti reduction} we have then
\begin{align*}
&\reg_{\wtP_{\,[1]}}(\wtR)=\reg_{\,\ov Q}(\ov R)=\ell+1\,.\qedhere
\end{align*}
\end{proof}

\begin{comment}
Adopt \cref{not:1} and assume that $n$ is even. Then $n=2\ell+2$ and $s= n-2$.
For $1\leq j\leq n-2$ we set
\begin{equation*}
    \alpha_j\coloneq \sum_{i=0}^{\ell-1}(-1)^{i+j-1}\textstyle \binom{n}{i}\binom{\ell +n-2+j-2i}{n-2}-\sum_{r=1}^{j-1}(-1)^{r+j}\binom{n-2+j-r}{j-r}\alpha_r,
\end{equation*}
and 
\begin{equation*}
\alpha_{n-1}\coloneq (-1)^n+\sum_{i=2}^{n-\ell} (-1)^{i}\alpha_{n-i}+\sum_{i=n-\ell+1}^{n-1}
(-1)^i\Big(\alpha_{n-i}+\textstyle\binom{n}{i}\Big).
\end{equation*}
\end{comment}

\section{Regularity in the general case}
 \label{sec:generic}

 In this section we develop a major ingredient for our investigation, namely \cref{generic reg},  which extends \cref{p:reg ex} to a general setting. We begin by establishing a precise definition of the term ``general''. 

Let $\kk$ be an infinite field of characteristic zero or positive characteristic greater than $n$ where $n\ge 2$ is a fixed integer and $Q=\kk[x_1,\dots,x_n]$ a standard graded polynomial ring. 

\begin{chunk} {\bf Parametrization.}\label{chunk:parameter}
\label{param} Consider the set
\begin{equation*}
    \Upsilon \coloneq \{(k, i,j)\mid 1\leq k \leq n+1, 1\leq i\leq j\leq n\}
\end{equation*} 
equipped with some total order, and set $N = |\Upsilon|-1 = (n+1)\binom{n+1}{2}-1.$ 
We write an element of the projective space $\mathbb P^N$ over $\kk$ as $\bsa=(a_\upsilon)_{\upsilon\in \Upsilon}$, where $a_{\upsilon}\coloneq a_{k ij}\in \kk$ for $\upsilon=(k,i,j)$. For each $\bsa\in \mathbb P^N$  consider the following set of $n+1$ quadrics in $Q$:
\begin{equation} 
 \label{def f}   
 \Big\{f_k=\sum_{1\leq i\leq j\leq n}a_{kij}x_ix_j\Big\}_{1\leq k \leq n+1}.
\end{equation}

Further we set 
\begin{equation*}
J_\bsa \coloneq (f_1,\dots f_n),\quad
I_\bsa \coloneq (f_1,\dots f_{n+1}), \quad\text{and}\quad
R_\bsa \coloneq Q/I_\bsa.
\end{equation*}

Also, fix $\widetilde \bsa\in \mathbb P^N$ such that $R_{\widetilde{\bsa}}=\widetilde R$, where $\widetilde R$ is as in \cref{not:3}.
\end{chunk}

\begin{definition}
\label{def: generic}
We say that a {\it general quotient $R$ of $Q$ defined by $n+1$ quadrics satisfies a property $\mathcal P$} if there exists a nonempty Zariski open subset $U$ of $\mathbb P^N$ such that $R_\bsa$ satisfies $\mathcal P$ for all $\bsa\in U$. 
\end{definition}

\begin{remark}\label{general} A general quotient $R$ of $Q$ defined by quadrics $f_1, \dots, f_{n+1}$ satisfies the following properties:
\begin{enumerate}
    \item $\{f_1, \dots, f_{n}\}$ is a regular sequence, and thus $R$ is an almost complete intersection;
    \item the complete intersection $Q/(f_1,\dots, f_n)$ satisfies the SLP;
    \item $f_{n+1}$ is a maximal rank element on $Q/(f_1,\dots,f_n)$.
\end{enumerate}
Property (1) follows from \cite[Theorem 1]{Froberg-Lofwal-ci}.  In characteristic zero, property (2) follows from the result of Stanley in \cite{Stanley80} and Watanabe in \cite{Wat} that monomial complete intersections have the SLP and from semicontinuity; see for example \cite[Page 81]{MMR03}. Similarly, in characteristic greater than $n$, property (2) follows from the analogous result \cite[Theorem 7.2]{cook} and from semicontinuity.  It follows from (2) that $Q/(f_1,\dots,f_n)$ has a quadratic maximal rank element and (3) now follows from semicontinuity. Observe that properties (1)-(3) hold for $\widetilde R$.  As such, the proofs of the above mentioned results indicate that we can choose an open set $U$ containing $\widetilde \bsa$ such that properties (1)-(3) hold for $R=R_\bsa$ with $\bsa\in U$. 
\end{remark}

\begin{terminology}
\label{RPA}
For the remainder of the paper, we adopt the following convention. In any statement that addresses the properties of a general quotient $R$ of $Q$ defined by $n+1$ quadrics, we assume that $R$ satisfies the properties (1) and (3) in \cref{general}, and we refer to the following rings 
\[
P=Q/(f_1, \dots, f_n) \qquad \text{and} \qquad A=Q/((f_1, \dots, f_n)\colon f_{n+1})
\]
as the \textit{complete intersection ring $P$ associated to $R$} and the \textit{Gorenstein ring $A$ associated to $R$}, respectively, where $f_1, \dots, f_{n+1}$ are defined as in \cref{chunk:parameter}. 
\end{terminology}

\begin{remark}\label{rem: general - hilb}
By \cref{general}, \cref{ex: wlp} and \cref{Hilb R and A},  for a general quotient $R$ of $Q$ by $n+1$ quadrics (more precisely, any $R$ satisfying (1) and (3) in \cref{general}) the following formulas are satisfied, where $A$ is the Gorenstein ring associated to $R$:
\begin{enumerate}
    \item $\fh_R(u)=\fh_{\wtR}(u)=\max \textstyle{\left\{ \binom{n}{u}-\binom{n}{u-2},0\right\}}$ for all $u\ge 0$;
    \item $\fh_A(u)=\fh_{\wtA}(u)=\min\left\{\textstyle \binom{n}{u},\binom{n}{u+2}\right\}$ for all $u\ge 0$.
\end{enumerate}
Additionally, when $n$ is even, Proposition 3.10, Example 3.11, Proposition 3.13, Example 3.14 in \cite{proceedings} give that $R$ is level with socle in degree $n-\ell-1$ when $n\ge 2$ and the defining ideal $G$ of $A$ is generated in degree $2$ and $\ell+1$ when $n\ge 3$. 
\end{remark}

\begin{notation}\label{not:general}
Now we establish notation to be used throughout the remainder of this section.  For $\bsga=(\gamma_1, \dots, \gamma_n)\in \mathbb N^n$, we set ${\bsx^\bsga}\coloneq x_1^{\gamma_1}\dots x_n^{\gamma_n}\in Q$ and $|\bsga|\coloneq\gamma_1+\cdots+\gamma_n$. For each $i\ge 0$ we set 
\[
\Gamma_i\coloneq\{\bsga\in \mathbb N^n\colon |\bsga|=i\}.
\]
For  $i,j$ with $1\le i\le j\le n$ there exists a unique $\bsbe\in \Gamma_2$ such that $x_ix_j=\bsx^\bsbe$, and we set $a_{k, \bsbe}\coloneq a_{k ij}$. In particular, we have equalities 
\begin{equation}\label{def f sum}
f_k=\sum_{\bsbe\in\Gamma_2}a_{k,\bsbe}\bsx^\bsbe\quad\text{for all}\quad 1\leq k\leq n+1.  
\end{equation}
In what follows, we will use the notation ${\bsx^\bsga}$ to also stand for the image of ${\bsx^\bsga}$ in a quotient of $Q$; this will be clear from the context. 
\end{notation} 

\begin{lemma}
\label{U2}
 Let $\wtR$ be as in \cref{not:3} and for each $i\ge 0$ let $\{\bsx^\bsla\}_{\bsla\in \Lambda_i}$ be a basis for the $\kk$-vector space $\wtR_i$. A general quotient $R$ of $Q$ defined by $n+1$ quadrics  has the properties:
\begin{enumerate}[\quad $(1)$]
\item For all $i\ge 0$ the set $\{\bsx^\bsla\}_{\bsla\in \Lambda_i}$ is a basis for $R_i$. 
\item Write $R= R_{\bsa}$, for some $\bsa\in\mathbb{P}^N$. Then for every $i\ge 0$ there exists an integer $m_i\ge 0$ and a homogeneous polynomial $q_i$ in $\bsa$ such that, when writing a monomial in $R_i$ as a linear combination of the basis $\{\bsx^\bsla\}_{\bsla\in \Lambda_i}$,  the coefficients are rational expressions in $\bsa$ with denominator $q_i$ and numerator a homogeneous polynomial  in $\bsa$ of degree $m_i$. 
\end{enumerate}
\end{lemma}
\begin{proof} 
Let $U_1\subseteq\mathbb P^N$ be a nonempty open set containing $\widetilde a$ such that $R_{\bsa}$ satisfies the properties of \Cref{rem: general - hilb}(1) for all $\bsa\in U_1$. Let $\bsa\in U_1$, set $R\coloneq R_{\bsa}$, and let  $f_1, \dots, f_{n+1}$ be the defining quadrics of $R_{\bsa}$, as in \eqref{def f sum}. Note that by  \Cref{rem: general - hilb}(1) we have that $R$ is Artinian and $R_i=0$ for $i>\reg_Q(\widetilde R)$. Thus, it suffices to assume $i\le \reg_Q(\widetilde R)$.

 Let $0\le i\le \reg_Q(\widetilde R)$. Since $\bsa\in U_1$,  we have $\rank_\kk(R_i)=\rank_\kk(\widetilde R_i)=|\Lambda_i|$. Thus, the set $\{\bsx^\bsla\}_{\bsla\in \Lambda_i}$ is a basis for $R_i$ if and only if it is a spanning set for $R_i$.  To check this,  for each $g\in Q_i$ we need to find $b_{\bsla}\in\kk$ and $p_k \in Q_{i-2}$ such that 
\begin{equation}
 \label{eq1}
 \sum_{\bsla\in\Lambda_i}b_\bsla \bsx^\bsla +\sum_{k=1}^{n+1}p_k f_k=g.
\end{equation}
If we write $p_k =\sum_{\bsal\in\Gamma_{i-2}}p_{k,\bsal}\bsx^\bsal$ and $g=\sum_{\bsga\in\Gamma_i}c_{\bsga}\bsx^{\bsga}$ with $p_{k, \bsal},c_{\bsga}\in \kk$, then using \eqref{def f sum}, the  equation \eqref{eq1} can be rewritten as follows:
\begin{equation}
\label{eq2}
\sum_{\bsla\in\Lambda_i}b_\bsla \bsx^\bsla +
\sum_{k=1}^{n+1} \sum_{\small\substack{\bsal\in\Gamma_{i-2}\\ \bsbe\in \Gamma_2}}p_{k,\bsal}a_{k, \bsbe}\bsx^{\bsal+\bsbe}=\sum_{\bsga\in\Gamma_i}c_{\bsga}\bsx^{\bsga}. 
\end{equation}
Equating the coefficients of each $\bsx^{\bsga}$ in \eqref{eq2}, we obtain a system of $|\Gamma_i|$ equations with unknowns $b_{\bsla}$ and $p_{k, \bsal}$, where $\bsla\in \Lambda_i$, $1\le k\le n+1$,  and $\bsal\in \Gamma_{i-2}$. Namely, for each $\bsga\in \Gamma_i$ we have an equation
\begin{equation}
\label{sys}
c_\bsga = \begin{cases} \,
b_\bsga + \displaystyle{\sum_{k=1}^{n+1} \sum_{\bsal+\bsbe=\bsga}p_{k,\bsal}a_{k,\bsbe},}&\ \text{if } 
\bsga\in\Lambda_i\\
\,\displaystyle{\sum_{k=1}^{n+1} \sum_{\bsal+\bsbe=\bsga}p_{k,\bsal}a_{k,\bsbe}},&\ \text{if } 
\bsga\in{\Gamma_i\setminus\Lambda_i}.\\
\end{cases}
\end{equation}
To describe the $|\Gamma_i|\times \big(|\Lambda_i|+(n+1)|\Gamma_{i-2}|\big)$ coefficient matrix ${\cA}$ of this system, we index the rows by $\bsga \in \Gamma_i$ and the columns by the names of the unknowns  $b_{\bsla}$ and $p_{k, \bsal}$. Namely, 
\begin{equation}\label{eq:entries}
\cA_{\bsga, b_{\bsla}}=\begin{cases} 1,&\text{if}\ \bsga=\bsla\\ 0,& \text{if}\ \bsga\not=\bsla\end{cases}\qquad\text{and}\qquad
\cA_{\bsga,p_{k,\bsal}}=\begin{cases}a_{k,\bsbe},&\text{if}\  \bsal+\bsbe=\bsga\\
0,&\ \text{otherwise}.\end{cases}
\end{equation}

Now, we have that $\{\bsx^\bsla\}_{\bsla\in \Lambda_i}$ is a spanning set of $R_i$ if and only if the system \eqref{sys} is consistent for all polynomials $g$. This is equivalent to the fact that the matrix $\cA$ has a nonzero maximal minor of size equal to the number of rows of $\cA$.

Let $\widetilde{\cA}$ represent the corresponding matrix for the ring $R_{\widetilde{\bsa}}=\wtR$. By hypothesis, we have that $\{\bsx^\bsla\}_{\bsla\in \Lambda_i}$ is a basis for $\wtR_i$, thus it follows from the preceding paragraph that $\widetilde{\cA}$ has a nonzero maximal minor of size equal to the number of rows of $\widetilde{\cA}$. Therefore, the number of rows of $\widetilde{\cA}$ is less than or equal to the number of columns, and the same holds for $\cA$ since its size is independent of $\bsa$. Note that by \eqref{eq:entries}, the minors of $\cA$ are polynomials in $\bsa$, so  we may choose a nonempty Zariski open set $V_i$ containing $\widetilde{\bsa}$ such that $\cA$ has a nonzero maximal minor when $\bsa\in V_i$. Now we conclude that $\{\bsx^\bsla\}_{\bsla\in \Lambda_i}$ is a spanning set, and thus a basis, of $R_i$ for all $\bsa\in U_1\cap V_i$. 

We set $U_2\coloneq \big(\bigcap_{i=0}^{r} V_i\big)\cap U_1$, where $r=\reg_Q(\widetilde R)$.  Since $\kk$ is infinite, a finite intersection of nonempty Zariski open sets is also nonempty, and hence the set $U_2$ is a nonempty Zariski open set. Therefore,  $R_\bsa$ satisfies property $(1)$ when $\bsa\in U_2$.   

In fact, $R_\bsa$ also satisfies property (2) whenever $\bsa\in U_2$. Indeed, let $i\ge 0, \bsa\in U_2,$ and $g$ be a monomial in $R_i$. To write $g$ in terms of the basis $\{\bsx^\bsla\}_{\bsla\in \Lambda_i}$, we need to solve the system \eqref{sys}. If this system has some free variables, we choose them to be zero, and then we solve the system using Cramer's rule. The conclusion follows from the description of $\cA$. 
\end{proof}

We now introduce additional notation and conventions that will be used in the proof of the main result of the section, \cref{generic reg}. 

\chunk\label{order} We make the convention that if  $\Gamma\subseteq \mathbb Z$ and  $(\mathcal A_u)_{u\in \Gamma}$ is a collection of pairwise disjoint ordered sets, then the induced order on the union $\bigcup_{u\in \Gamma}\mathcal A_u$ is defined such that it extends the order on each of the sets $\mathcal A_u$, and, if $A\in \mathcal A_u$ and $A'\in \mathcal A_v$ then  $A<A'$ if and only if $u<v$. 

\chunk  
\label{constr res}
Consider the hypersurface  ${\Pu{1}}=Q/(f_1)$ defined by a nonzero quadric 
\[
f_1=\sum\limits_{1\leq i\leq j\leq n} a_{1ij}x_ix_j \quad\text{with}\quad a_{1ij}\in \kk.
\]
Let $(K,\partial^K)$ denote the Koszul complex on $x_1, \dots, x_n$ over  $\Pu{1}$. Note that $K$ is a differential graded (DG-) algebra, and it comes equipped with an internal grading. We write $(K_i)_j$ for the component in homological degree $i$ and internal degree $j$. Fix an ordered basis $E_1,\dots, E_n$ of $K_1$ with $\partial^{K}_1(E_i)=x_i$.    
For every $0\leq i \leq n$, define
\[
\mathcal B_i=\{(b_1,\dots, b_i)\in \mathbb Z^i\colon 1\leq b_1< \dots < b_i\leq n\}.
\]
If $B\in \mathcal B_i$, we write $E_B\coloneq E_{b_1}\wedge\dots\wedge E_{b_i}\in (K_i)_{i}$, where $E_B=1$ if $B$ is the empty tuple. Then an ordered basis of $K_i$ consists of all $E_B$ with $B\in \mathcal B_i$, ordered using the  lexicographic order on $\mathcal B_i$. For convenience, we set $\mathcal B_i=\emptyset$ when $i<0$. 

Recall that for all $1\leq i\leq n$, for $B\in\mathcal{B}_i$, the $i$-th differential of $K$ is given by 
\begin{equation*}
    \partial_i^{K}(E_B)=\sum_{j=1}^i(-1)^jx_{b_j}E_{B\setminus b_j}.
\end{equation*}

Define a cycle element $z\in (K_1)_2$  by
\[
z=\sum\limits_{1\leq i\leq j\leq n} a_{1ij}x_iE_{j}=\sum_{j=1}^n\Big(\sum_{1\leq i\leq j} a_{1ij}x_i\Big)E_{j}.
\] 

For each $i\ge 0$ we define a homomorphism $\zeta_i\colon K_i(-2)\to K_{i+1}$ of graded $\Pu{1}$-modules by

\begin{equation*}
\label{def zeta}
    \zeta_i(E_B)\coloneq z\wedge E_B= \sum_{j\not\in B}(-1)^{|j|_B}\Big(\sum_{1\leq i\leq j} a_{1ij}x_i\Big)E_{B \cup \{j\}},
\end{equation*}
where $B=(b_1, \dots, b_i)\in \mathcal B_i$. Here,  $B \cup \{j\}$ denotes the ordered tuple obtained by inserting $j$ into $B$ and $|j|_B$ denotes the largest integer $k$ such that $j>b_k$, and if no such integer exists, we set $|j|_B=0$. By \cite[Theorem 1]{NV}, one can describe the graded minimal free resolution of $\kk$ over ${\Pu{1}}$, denoted  $(F,\partial^F)$, as follows.
For all $i\in\BZ$ set  
\begin{equation*}
\label{e:Fdefn}
F_i\coloneq K_i\oplus K_{i-2}(-2)\oplus K_{i-4}(-4)\oplus\dots\,.
\end{equation*}
For each $i\ge 0$, consider an ordered basis of $F_i$ consisting of the elements $E_B$ (with internal degree shifted accordingly) with $B\in \bigcup_{u\ge 0}\mathcal B_{i-2u}$; the order on the union is induced by the lexicographic order of each $\mathcal B_j$ as in \ref{order}. Define the graded $\Pu{1}$-module homomorphisms
\begin{equation*}
    \partial_i^F\colon F_i\to F_{i-1}\quad\text{given by}\quad  
    \partial_i^F\coloneqq\begin{pmatrix}
\partial^K_{i} & \zeta_{i-2} & 0 & 0 & 0 & \cdots \\[0.2cm]
0 & \partial^K_{i-2} & \zeta_{i-4} & 0 & 0 & \cdots \\[0.2cm]
0 & 0 & \partial^K_{i-4} & \zeta_{i-6} & 0 & \cdots \\[0.2cm]
\vdots & \vdots  & \vdots & \ddots & \ddots &
\end{pmatrix},
\end{equation*}
as block matrices with respect to our fixed bases. Note that each entry of $\partial_i^F$ has degree $1$.

\begin{remark}
 \label{rmk: res k} 
 Observe that for each $i\ge 0$ we have an isomorphism of graded $\Pu{1}$-modules 
 \[
 K_i\cong (\Pu{1})^{\binom{n}{i}}(-i) \quad\text{and}\quad F_i\cong \Pu{1}^{\beta_i}(-i),\qquad \text{where}\quad  \beta_i\coloneq\sum_{j\ge 0}\binom{n}{i-2j}
\]
for all $i\geq 0$.  Moreover, since $K_i=0$ for all $i\geq n+1$, the resolution $(F,\partial^F)$ is periodic with $F_i=F_{i+2}$  for all $i\geq n-1$ and $\partial_i^F=\partial^F_{i+2}$ for all $i\geq n$. 
\end{remark}

\begin{theorem}
\label{generic reg}
Let $n\ge 2$ be an integer and set $\ell=\textstyle{\bigl\lfloor\frac{n-2}{2}\bigr\rfloor}$. Let $\kk$ be  an infinite field of characteristic zero or greater than $n$ and $Q=\kk[x_1, \dots, x_n]$. 
A general quotient $R$ of $Q$ defined by $n+1$ quadrics $f_1, \dots, f_{n+1}$ has the property:
\[
\reg_{\Pu{1}}(R)=\ell+1,
\]
where $\Pu{1}=Q/(f_1)$. Consequently, $\reg_P(A)=\ell$ and $\reg_P(R)=\ell+1$, where $P$ is the complete intersection ring and $A$ is the Gorenstein ring associated to $R$.
\end{theorem}

\begin{proof} 
Let $U$  be the nonempty open set on which the properties from \cref{U2} hold.  Let $\bsa\in U$ and set $R\coloneq R_\bsa$.

When $n$ is even, the  statement is given by \cref{reg}(2). When $n$ is odd, \cref{reg}(2)  gives $\ell +1 \le \reg_{\Pu{1}}(R)\le \ell+2$.  In this case, it remains to show  that, after restricting $U$ to a possibly smaller nonempty open set, we have
\begin{equation}
\label{Tor 0}
    \Tor_i^{\Pu{1}}(R,\kk)_{i+\ell+2}=0\quad\text{for all } i\geq 0.\quad 
\end{equation}
Let $(F,\partial^F)$ be  the graded minimal free resolution  of $\kk$ as a $\Pu{1}$-module, as described in \Cref{rmk: res k} and set $\overline F=F\otimes_{\Pu{1}}R$. 

Fix $i\ge 0$. We have  $\Tor_i^{\Pu{1}}(R,\kk)_{i+\ell+2}=\HH_{i}(\overline F)_{i+\ell+2}$ and for all $u\ge 0$: 
\begin{equation}
\label{RF}
(\overline F_u)_{i+\ell+2}\cong \left(R^{\beta_u}(-u)\right)_{i+\ell+2}=(R_{i+\ell+2-u})^{\beta_u},
\end{equation}
with $\beta_u$ as in \cref{rmk: res k}. 
The module $\Tor_i^{\Pu{1}}(R,\kk)_{i+\ell+2}$ can be computed as the homology in the middle spot of the sequence
\begin{equation*}
(\overline F_{i+1})_{i+\ell+2}\xra{d_{i+1}} (\overline F_i)_{i+\ell+2}\xra{d_{i}}(\overline F_{i-1})_{i+\ell+2},
\end{equation*}
where $d_i$ and $d_{i+1}$ are the restrictions to internal degree $i+\ell+2$ of the differential of $\overline F$. By \eqref{RF}, the codomain of $d_i$ is isomorphic to $(R_{\ell+3})^{\beta_u}$. By \ref{Hilb R and A}(1), $R_{\ell+3}=0$, hence $d_i=0$. Thus,  \eqref{Tor 0} is equivalent to  \begin{equation}
\label{surj d}
    d_{i+1}\text{ is surjective for all $i\geq 0$. }
\end{equation}
For each $v\ge 0$, we work with the choice of basis for the free $R$-module $F_{v}$ described in \ref{constr res}, and we use the induced basis on $\overline{F}_v$. This basis lives in internal degree $v$, and hence for $j\ge 1$ the $\kk$-vector space  $(\overline F_{v})_{v+j}$ has a basis consisting of the elements $\bsx^\bsla E_B$ with $\bsla\in \Lambda_{j}$ and $B\in \bigcup_{u\ge 0}\mathcal B_{v-2u}$, where $\Lambda_{j}$ is chosen as in \cref{U2} so that $(\bsx^\bsla)_{\bsla\in \Lambda_j}$ is a basis for $R_j$. This basis has size $\beta_u|\Lambda_j|$ and its order is induced by the order on the index set $\bigcup_{u\ge 0}(\mathcal B_{v-2u}\times \Lambda_j)$, where the Cartesian products are ordered lexicographically, using the lex order on each component, and the union has the induced order described in \ref{order}.

Let $d_{i+1}[u,v]$ denote the component of $d_{i+1}$ with
\[
d_{i+1}[u,v]\colon \Span_\kk\left(\bsx^\bsla E_B\colon (B,\bsla)\in \mathcal B_{i+1-2u}\times \Lambda_{\ell+1}\right)\to \Span_\kk\left(\bsx^\bsla E_B\colon  (B,\bsla)\in \mathcal B_{i-2v}\times \Lambda_{\ell+2}\right).
\]

Observe that $d_{i+1}[u,v]=0$ for all $u,v$ such that $v\notin\{ u, u-1\}$. When $v=u$, the map $d_{i+1}[u,u]$ is a restriction of the differential $\partial_{i+1-2u}^K$ and when $v=u-1$ the map $d_{i+1}[u,v]$ is a restriction of $\zeta_{i+1-2u}$. We use overbars to denote these restrictions, regarded as maps of $\kk$-vector spaces, so that we have
\begin{equation}
\label{def d i+1}
    d_{i+1}=\begin{pmatrix}
\overline\partial_{i+1}^K & \overline\zeta_{i-1} & 0 & 0 & 0 & \cdots \\[0.2cm]
0 & \overline \partial_{i-1}^K & \overline\zeta_{i-3} & 0 & 0 & \cdots \\[0.2cm]
0 & 0 & \overline \partial_{i-3}^K & \overline\zeta_{i-5} & 0 & \cdots \\[0.2cm]
\vdots & \vdots  & \vdots & \ddots & \ddots &
\end{pmatrix}.
\end{equation}
For each $k\in [n]$, let $\chi_k$ denote the multiplication map $R_{\ell+1}\xra{x_k\cdot}R_{\ell+2}$, where the bases of $R_{\ell+1}$ and $R_{\ell+2}$ are chosen as indicated above. 
We identify all maps above with their matrices with respect to the chosen bases. Observe that the matrices $\partial_j^K$ and $\zeta_j$ have entries that are linear in the variables $x_1, \dots, x_n$, and thus the matrices $\overline \partial_j^K$ and $\overline \zeta_j$ can be obtained by replacing each variable $x_k$ in the entries of $\partial_j^K$ and $\zeta_j$ with the matrix $\chi_k$. Note that the matrix  $d_{i+1}$ has size $\beta_i|\Lambda_{\ell+2}|\times \beta_{i+1}|\Lambda_{\ell+1}|$.  

Using \cref{U2} and the notation therein, we see that for each $k\ge 0$ the nonzero entries of the matrix $\chi_k$ are rational expressions in $\bsa$ with denominator $q_{\ell+2}$ and numerator a homogeneous polynomial in $\bsa$ of degree $m_{\ell+2}$. Further, using the descriptions of the maps $\partial^K_j$ and $\zeta_j$, we see that the nonzero entries of the matrix $q_{\ell+2}\overline \partial^K_j$ are polynomials in $\bsa$ of degree $m_{\ell+2}$ and the nonzero entries of the matrix $q_{\ell+2}\overline \zeta_j$ are polynomials in $\bsa$ of degree $m_{\ell+2}+1$. 

The proof of \cref{U2} shows that $\Tilde \bsa\in U$ where $\Tilde{\bsa}$ is such that $R_{\Tilde{\bsa}}=\wtR$. Let $\widetilde{d_{i+1}}$ represent the corresponding matrix. By \cref{p:reg ex} we have that $\Tor_i^{\wtP_{\,[1]}}(\wtR,\kk)_{i+\ell+2}=0$, where $\wtP_{\,[1]}=Q/(x_1^2)$, and thus $\widetilde{d_{i+1}}$ is surjective. In particular, the number of rows of the matrix $\widetilde{d_{i+1}}$ is less than or equal to the number of columns, and the same holds for $d_{i+1}$ since its size is independent of $\bsa$. Surjectivity of the map $d_{i+1}$ is then equivalent to the matrix $d_{i+1}$ having a nonzero maximal minor. 

A cofactor expansion shows that the determinant of any maximal minor of $q_{\ell+2}d_{i+1}$ is a homogeneous polynomial in $\bsa$. Thus, we may choose a Zariski open set $W_{i+1}\subseteq\mathbb P^N$ such that the matrix $d_{i+1}$ has a nonzero maximal minor when $\bsa\in W_{i+1}$. Since $\widetilde{d_{i+1}}$ is surjective, we may choose this set to contain $\widetilde \bsa$. By \Cref{rmk: res k}, the resolution $F$ becomes periodic after $n+1$ steps, so for all $i\geq 0$, the matrix $d_{i+1}$ has a nonzero maximal minor when $\bsa$ is chosen in the Zariski open subset $V$ of $\mathbb P^N$, where $V\coloneq U\cap \bigcap_{i=0}^{n+2} W_{i+1}$.  Thus we have shown \eqref{surj d}, and consequently \eqref{Tor 0}.  In conclusion, we see that $\reg_{\Pu{1}}(R)\le \ell+1$ when $\bsa\in V$. 

The last statement regarding regularity over $P$ follows now directly from \Cref{p:Betti-over-P}(3b) and (3a), respectively.
\end{proof}

Below we summarize some key properties satisfied by a general quotient by $n+1$ quadrics that we will need in the proofs of our results throughout the remainder of the paper.  In fact, the hypotheses of every subsequent result about a general quotient by $n+1$ quadrics can be replaced by these properties. 

\begin{remark}
\label{d:newgeneric}
A  general quotient $R$ of $Q$ by $n+1$ quadrics $f_1, \dots, f_{n+1}$ satisfies the following properties (with the notation of \cref{not:1}): 
\begin{enumerate}[\quad (a)]
 \item  $f_1, \dots, f_n$ is a regular sequence and $f_{n+1}$ is a maximal rank element of $P$; 
 \item $\reg_{\Pu{1}}(R)=\ell+1$ (and hence $\reg_P(A)=\ell$ and $\reg_P(R)=\ell+1$ by \cref{p:Betti-over-P}).
 \end{enumerate}
Indeed, let $U_1$ be an open set such that $R_\bsa$ satisfies the conclusions of \cref{generic reg} for all $\bsa\in U_1$ and let $U_2$ be an open set such that $R_\bsa$ satisfies (1) and (3) of \cref{general} for all $\bsa\in U_2$.   Thus the desired properties hold when  $R=R_\bsa$ with $\bsa\in U_1\cap U_2$. 
\end{remark}

%%%%%%%%%%%%%%%%%%%%%%%%%%%%%%%%%%%%%%%%%%%%%%%%%%%%%%%%%%%%%%%%%%%%%%%%%%%%%%%%%%%%%%%%%%%%%%%%%%%%%%%%%%%%%%%%%%%%%%%%%
\section{Computation of Poincar\'e series}\label{sec:Poincare}

In this section, let $\kk$ be an infinite field of characteristic zero or positive characteristic greater than $n$ where $n\ge 2$ is a fixed integer and $Q=\kk[x_1,\dots,x_n]$ a standard graded polynomial ring. We continue the study of a general quotient $R$ of $Q$ defined by $n+1$ quadrics and of the associated  Gorenstein ring $A$; see \Cref{def: generic} and \Cref{d:newgeneric}. The main result is \Cref{t:generic}, which establishes the existence of Golod homomorphisms (see \Cref{def:golod}) from a graded complete intersection ring onto the rings $R$ and $A$. Thus all finitely generated graded modules over $R$, respectively $A$, have rational Poincar\'e series sharing a common denominator; see \cite[Proposition 5.18]{AKM}. This property is known to hold for absolutely Koszul algebras, and can be regarded as a generalization of the absolutely Koszul property. 
Furthermore, we provide explicit formulas for the Poincar\'e series of the residue field over the rings $R$ and $A$.

First we establish a lemma that is needed for the proof of the theorem. 

\begin{lemma}
\label{l:ab}
Adopt the notation and hypotheses in \cref{not:1} and set 
\begin{equation*}
\label{def:ab}
a_i\coloneq\beta_{i,i+\ell}^P(A)\quad\text{and}\quad b_i\coloneq\beta_{i,i+\ell+1}^P(A)\quad\text{for all}\ i\geq 0.
\end{equation*}
If $f_{n+1}$ is a quadratic maximal rank element of $P$, then the following equality holds:
\begin{equation}
\label{e:2strands}
\Po_A^P(t,u)=1+u^\ell\sum_{i\ge 1} (a_i+b_iu)(tu)^i\,.
\end{equation}
In particular, if $\reg_P(A)= \ell$,  that is,  $b_i=0$ for all $i\geq 1$, then the following equality holds:
\begin{align}
\Po_A^P(t,u)=1+\frac{(1-tu)^n-\Hilb_A(-tu)}{(-1)^{\ell+1}t^\ell (1-tu)^n}\label{e:linearhilb}.
\end{align}
\end{lemma}

\begin{proof} By \cref{p:Betti-over-P}(2), we know that $\beta_{i,j}^P(A)=0$ when $j\not\in\{i+\ell,i+\ell+1\}$ and $i>0$.
Thus, \eqref{e:2strands} follows from the definition of Poincar\'e series. 

Next, assume that $\reg_P(A)=\ell$. Then $b_i=0$ for all $i\ge 1$, so by 
\eqref{e:2strands} we have
\begin{align}
\Po_A^P(t,u)&=1+u^\ell\sum_{i\ge 1} a_i(tu)^i. \label{e:linear}
\end{align}
Notice that the ideal $L=\ann_P(f_{n+1})$ of $P$ has an $(\ell+1)$-linear resolution, since $A\cong P/L$.
Thus by \cref{a-linear}, the Poincar\'e series of the  $P$-module $L$ is given by
\[
\Po_{L}^P(t,u)=\frac{\Hilb_{L}(-tu)}{(-t)^{\ell+1}\Hilb_P(-tu)}\,.
\]
We also have equalities: 
\[
\Hilb_{L} (t)=\Hilb_P(t)-\Hilb_A(t)=(1+t)^n-\Hilb_A(t)\,,
\]
where the second equality follows since $P$ is a complete intersection defined by $n$ quadrics.
Thus, \eqref{e:linearhilb} follows by observing that $\Po_A^P(t,u)-1=t\cdot\Po_{L}^P(t,u)$. 
\end{proof}

Now we prove the main theorem of this section.  In part (1) of the theorem we assume that $n\ge 4$; we address smaller values of $n$ in \cref{small n}.  Note that the conclusion that $R$ is not Koszul when $n\ge 4$ in part (2) also follows from \cite[Theorem 7.1]{Froberg-Lofwal}; see also \cite[Lemma 4.2]{McCullough-Seceleanu}.

\begin{theorem}
\label{t:generic}
Let $n\geq 2$ be an integer and set $\ell=\textstyle{\bigl\lfloor\frac{n-2}{2}\bigr\rfloor}$.  Let $\kk$ be an infinite field of characteristic zero or greater than $n$ and $Q=\kk[x_1, \dots, x_n]$. 
A general quotient $R$ of $Q$ defined by $n+1$ quadrics $f_1, \dots, f_{n+1}$ has the following properties, where $P$ is the complete intersection ring and $A$ is the Gorenstein ring associated to $R$. 
 \begin{enumerate}[\quad $(1)$]
\item If $n\ge 4$, then the natural projection $P\to A$ is Golod and the Poincar\'e series of $\kk$ over $A$ is 
\begin{align}
 \Po_\kk^A(t,u)
 \label{e: kA over Hilb}
               &=\frac{(-t)^{\ell-1}}{\Hilb_A(-tu)+(1-tu)^n((-t)^{\ell-1}-1)}.
\end{align}
\vspace{1pt}

Furthermore, $A$ is absolutely Koszul if $n=4$ or $n=5$ and is not Koszul when $n\ge 6$.

\item The natural projection $P\to R$ is Golod and the Poincar\'e series of $\kk$ over $R$ is
\begin{align}
 \Po_\kk^R(t,u)
 \label{e: kR over Hilb}
 &=\frac{(-t)^{\ell}}{\Hilb_R(-tu)+(1-tu)^n((-t)^{\ell}-1)(1-t^2u^2)}.
 \end{align}
 \vspace{1pt}
 
Furthermore, $R$ is absolutely Koszul if $n=2$ or $n=3$ and is not Koszul when $n\ge 4$.
\end{enumerate}
\end{theorem}

\begin{proof}  

We may assume that $R$ satisfies the properties listed in \cref{d:newgeneric}. In particular, $P$ is a complete intersection, $f_{n+1}$ is a maximal rank element on $P$, $\reg_P(A)=\ell$, and $\reg_P(R)=\ell+1$.

(1): Since $n\ge 4$, we see that $\ell\ge 1$, and hence the kernel of the natural projection map $\pi\colon P\to A$ has generators in degrees at least $2$ by \cref{Hilb R and A}(2b). 
By \cref{l:powers-zero-map} in the Appendix, the map 
\[
\Tor_i^P(A/\n^2,\kk)\to \Tor_i^P(A/\n,\kk)
\]
induced by the canonical projection $A/\n^2\to A/\n$ is zero, where $\n$ denotes the maximal homogeneous ideal of $A$. By \cref{p:Betti-over-P}(2) we have $\beta_{i,j}^P(A)=0$ when $j-i<\ell$. By \Cref{p:Betti-over-P}(3b) we have 
$\reg_P(A)=\ell$, and thus $\beta_{i,j}^P(A)=0$ when $j-i\geq \ell+1$.
Since $\ell\geq 1$, we have $2\ell\ge \ell+1$ and the hypotheses of \cref{p:Golod-I}(I) hold with $c=1$ and $b=\ell$, yielding that $\pi$ is Golod, as desired. 

Since $\pi$ is a Golod homomorphism, the Poincar\'e series of $\kk$ over $A$ is given by
\begin{align}\label{e:levin}
\Po_\kk^A(t,u)&=\frac{\Po_\kk^P(t,u)}{1-t(\Po_A^P(t,u)-1)}
\end{align}
by \cite[Theorem 1.6]{Lev}. 
    Since $\reg_P(A)=\ell$, we apply \cref{l:ab} and substitute \eqref{e:linearhilb} and the equality $\Po_{\kk}^P(t,u)=(1-tu)^{-n}$ into \eqref{e:levin} to obtain:
\begin{align*}
\Po_{\kk}^A(t,u)&=\frac{1}{(1-tu)^n\left(1-\frac{(1-tu)^n-\Hilb_A(-tu)}{(-1)^{\ell+1}t^{\ell-1}(1-tu)^n}\right)}\\
&=\frac{(-t)^{\ell-1}}{\Hilb_A(-tu)+(1-tu)^n((-t)^{\ell-1}-1)},
\end{align*}
yielding \eqref{e: kA over Hilb}.

When $n=4$ or $n=5$, we have $\ell=1$, and thus $\reg_P(A)=1$. Combining this with the fact that $\Ker(\pi)$ has generators in degrees at least $2$, we have that $\Ker(\pi)$ has a $2$-linear resolution. Using \cref{prop:abs-koszul}, it follows that $A$ is absolutely Koszul. 

The ideal $G$ defining $A$ is generated in degrees $2$ and $\ell+1$ by \cref{Hilb R and A}(2b). When $n\ge 6$, we have $\ell\ge 2$, hence $G$ is not quadratic, implying $A$ is not Koszul. 
  
(2):  By \cref{l:powers-zero-map} in the Appendix, the map 
\[
\Tor_i^P(R/\m^2,\kk)\to \Tor_i^P(R/\m,\kk)
\]
induced by the canonical projection map $R/\m^2\to R/\m$ is zero, where $\m$ denotes the maximal homogeneous ideal of $R$.   By \cref{p:Betti-over-P}(1) and \cref{generic reg}, when $n\ge 4$ the hypotheses of \cref{p:Golod-I}(II) hold with $c=1$, $b=\ell+1$, $i_0=1$, and $j_0=2$ and when $2\le n\le 3$  the hypotheses of \cref{p:Golod-I}(I) hold with $b=1$, yielding the natural projection map $P\to R$ is Golod, as desired. 

Since the natural projection map $P\to R$ is a Golod homomorphism, the Poincar\'e series of $\kk$ over $R$ is given by
\begin{align}\label{e:levinR}
\Po_\kk^R(t,u)&=\frac{\Po_\kk^P(t,u)}{1-t(\Po_R^P(t,u)-1)}
\end{align}
by \cite[Theorem 1.6]{Lev}. By the exact sequence \eqref{ses APR} in \cref{APR}, we have
\begin{align}
\Po_R^P(t,u)-1&=tu^2\cdot \Po_A^P(t,u). \label{e:p R over P}
\end{align}

Since $\reg_P(A)=\ell$, we apply \cref{l:ab} and substitute \eqref{e:p R over P}, \eqref{e:linearhilb} and the equality $\Po_{\kk}^P(t,u)=(1-tu)^{-n}$ into \eqref{e:levinR} to obtain:
\begin{align*}
\Po_{\kk}^R(t,u)&=\frac{1}{(1-tu)^n\left(1-t^2u^2-\frac{t^2u^2(1-tu)^n-t^2u^2\Hilb_A(-tu)}{(-1)^{\ell+1}t^{\ell}(1-tu)^n}\right)}\\
&=\frac{(-1)^{\ell+1}t^{\ell}}{(1-tu)^n((-1)^{\ell+1}t^{\ell}-(-1)^{\ell+1}t^{\ell+2} u^2-t^2u^2)+t^2u^2\Hilb_A(-tu)}.
\end{align*}

Next we rewrite the Hilbert series of $A$ in terms of the Hilbert series of $R$ and $P$ using \eqref{hilb APR} in \cref{APR} and use the fact that $\Hilb_P(-tu)=(1-tu)^n$
to get
\begin{align}
    t^2u^2\Hilb_A(-tu)=(1-tu)^n-\Hilb_R(-tu).
\end{align}
Thus we have equalities 
\begin{align*}
\Po_{\kk}^R(t,u)&=\frac{(-1)^{\ell+1}t^{\ell}}{(1-tu)^n((-1)^{\ell+1}t^{\ell}-(-1)^{\ell+1}t^{\ell+2} u^2-t^2u^2+1)-\Hilb_R(-tu)} \\
&=\frac{(-t)^{\ell}}{\Hilb_R(-tu)+(1-tu)^n((-t)^{\ell}-1)(1-t^2u^2)}, 
\end{align*}
yielding \eqref{e: kR over Hilb}.

 In particular, when $n\ge 4$, we have $\ell\ge 1$, and thus $P_{\kk}^R(t,u)\ne(\Hilb_R(-tu))^{-1}$ and hence $R$ is not Koszul.
\end{proof}

\begin{remark}
\label{less}
As indicated in the proof, the conclusions of \cref{t:generic} hold when assuming that $R$ satisfies the properties listed in \cref{d:newgeneric}, that is, in the notation of \cref{not:1}, we have:
\begin{enumerate}[\quad(a)]
\item $f_1, \dots, f_n$ is a regular sequence and $f_{n+1}$ is a maximal rank element on $P$;
\item $\reg_P(A)=\ell$ (and hence $\reg_P(R)=\ell+1$ by \eqref{e:reg}). 
\end{enumerate}
Furthermore, when $n$ is even, only hypothesis (a) is needed since by \cref{reg}(3) we have (a)$\Rightarrow$(b). Thus  all conclusions of \cref{t:generic} hold when $n$ is even and (a) is satisfied.   

On the other hand, when $n$ is odd, \cref{reg} only gives bounds on $\reg_P(A)$, and thus we no longer know that the implication (a)$\Rightarrow$(b) holds. Nonetheless, one can see that the conclusion that $P\to A$ is Golod in \cref{t:generic} also holds only under the assumption (a) when $n=4$ or $n>5$; the proof of this statement follows along the same lines, using \cref{p:Golod-I}. 
Additionally, as noted in \cite[Remark 3.6]{proceedings}, condition (a) is equivalent to formula \eqref{hilbertA} for the Hilbert function of $A$, and in some cases it is known that existence of Golod homomorphisms is a consequence of the form of the Hilbert function alone; see \cite{Rossi-Sega} and \cite{Kustin-Sega-Vraciu}.  Thus it seems reasonable to expect that only (a) is a necessary hypothesis.
\end{remark}

Since we know explicit formulas for the Hilbert series of $R$ and $A$ from \cref{Hilb R and A}, the formulas \eqref{e: kA over Hilb} and \eqref{e: kR over Hilb} for the Poincar\'e series in \cref{t:generic} can be explicitly calculated.  Alternatively, one can compute the Poincar\'e series using the Betti numbers $a_i=\beta_{i,i+\ell}^P(A)$ and $b_i=\beta_{i,i+\ell+1}^P(A)$ for all $i\ge 0$. This way of writing the formulas will come in handy in \cref{sec:rate}. 

\begin{corollary}\label{pseries}
Assume the hypotheses of \cref{t:generic} hold (more precisely, the properties of \cref{d:newgeneric} hold). Then 
\begin{align}
\Po_\kk^A(t,u)&=\frac{1}{(1-tu)^n \bigl(1-tu^\ell\sum_{i\ge 1}a_i(tu)^i)\bigr)} \label{e:k-over-A-general}\\
\Po_\kk^R(t,u) &=\frac{1}{(1-tu)^n \bigl(1-t^2u^2-t^2u^{\ell+2}\sum_{i\ge 1}a_i(tu)^i)\bigr)} \label{e:k-over-R}.
\end{align}
\end{corollary}

\begin{proof}  
\cref{t:generic} gives that the natural projections $P\to A$ and $P\to R$ are Golod. The fact that $P\to A$ is Golod gives
\begin{align}
 \label{e:k-over-A}
 \Po_\kk^A(t,u)&=\frac{1}{(1-tu)^n \bigl(1-tu^\ell\sum_{i\ge 1}(a_i+b_iu)(tu)^i)\bigr)}.
\end{align}
Indeed, this follows from substituting \eqref{e:2strands} into \eqref{e:levin}. 

Since $\reg_P(A)=\ell$, we have that $b_i=0$ for all $i\ge 0$ by \cref{l:ab}, and hence \eqref{e:k-over-A-general} holds. The equality \eqref{e:k-over-R} follows from substituting \eqref{e:linear} and \eqref{e:p R over P} into \eqref{e:levinR}. 
\end{proof}

\begin{remark}
\label{small n} 
In \cref{t:generic}(1), we require that $n\ge 4$, since the ideal $G$ defining $A$ has generators in degree $1$ otherwise. However, for small values of $n$ the ring $A$ is easy to describe. 

When $n=2$, as noted in \cite[Remark 3.2]{proceedings}  we have
\[I=(x_1,x_2)^2\quad\text{and}\quad R=\kk[x_1,x_2]/(x_1,x_2)^2 \quad\text{and}\quad A=\kk.\]
Thus $A$ is trivially Golod. Note that in this case $R=Q/\m^2$ is also Golod by \cite[Proposition 5.2.4]{Avr}.  

When $n=3$, it follows from \cref{Hilb R and A} that $\fh_A= (1,1)$ and hence $A\cong \kk[x]/(x^2)$ is a hypersurface, which is again Golod by \cite[Proposition 5.2.4]{Avr}. In this case, the almost complete intersection ring $R$ is not a Golod ring; see for example \cite{CVaci}.

 In both cases, note that $A$ is Koszul because it is either a field or defined by a quadratic monomial ideal; see \cite{Froberg-75}.  In particular, since $\ell=0$ in these cases, formula \eqref{e: kA over Hilb} does not hold. Also, by \cite[Proposition 5.8]{Herzog-Iyengar} we see that the projection $P\to A$ is not Golod, since $A$ is Koszul and $\reg_P(A)=\ell\ne 1$. On the other hand, the projection $Q\to A$ is Golod, since $A$ is Golod. In particular, the conclusion that $A$ is absolutely Koszul still holds by \cite[Theorem 5.9]{Herzog-Iyengar}.
\end{remark}

In view of \cite[Proposition 5.18]{AKM}, existence of a Golod homomorphism under the hypotheses of \cref{t:generic} yields the following. 

\begin{corollary}\label{c:rationalPoincare}
Assume the hypotheses of \cref{t:generic} hold. Then there exist polynomials $d_R(t), d_A(t)\in \mathbb Z[t]$ such that for every finitely generated graded $R$-module $M$, respectively $A$-module, we have $d_R(t)\Po_M^R(t)\in \mathbb Z[t]$, respectively  $d_A(t)\Po_M^A(t)\in \mathbb Z[t]$. 
\end{corollary}

\begin{remark}
With $R$ as in \cref{t:generic} and $n\ge 4$ (so that $R$ is not Koszul), observe that the formula \eqref{e: kR over Hilb} can be rewritten as: 
\begin{align*}
\frac{1}{\Po_R^\kk(t,u)}&=(-t)^{\ell}\Hilb_R(-tu)+\frac{(1-t^2u^2)^{n+1}}{(1+tu)^n}(1-(-t)^\ell)\\
&=(-t)^{\ell}\Hilb_{R}(-tu)+\frac{(1-(-t)^\ell)}{\Hilb_{R^!}(tu)}\,,
\end{align*}
where $R^!$ is the {\it Koszul dual} of $R$, which is the subalgebra of the Yoneda algebra $\Ext_R(\kk,\kk)$ generated by its elements of degree $1$. Indeed, the fact that the expression $(1+tu)^n(1-t^2u^2)^{-n-1}$ is the Hilbert series of $R^!$  follows from 
\cite[Theorem 2.5, Theorem 2.6]{Lofwall}, in view of the fact that the cubics $x_jf_i$ with $1\le j\le n$ and $1\le i\le n+1$ are linearly independent, which can be deduced from the expression of $\fh_R(3)$ in \eqref{hilbR}.  

Consequently, we see that formula (B.2) in \cite[Theorem B.4]{Roos2} is satisfied with $t=\ell+2$. This formula is proved as a consequence of a property called $\mathcal L_t$, which can be viewed as an extension of the Koszul property; see also \cite{Roos1}. 
Thus, although we do not explore the issue further in this paper, it seems reasonable to expect that $R$ satisfies the property $\mathcal L_{\ell+2}$; this is in fact observed in \cite{Roos1} to hold when $n\le 7$. Similarly, we see that the ring $A$ satisfies formula (B.2) in \cite[Theorem B.4]{Roos2} with $t=\ell+1$, leading to the expectation that $A$ satisfies $\mathcal L_{\ell+1}$. 
\end{remark}

\section{Rate computations}\label{sec:rate}

In this section, let $\kk$ be an infinite field of characteristic zero or positive characteristic greater than $n$ where $n\ge 2$ is a fixed integer and $Q=\kk[x_1,\dots,x_n]$ a standard graded polynomial ring.  We continue the study of a general quotient $R$ of $Q$ defined by $n+1$ quadrics and of the associated Gorenstein ring $A$ (see \Cref{def: generic} and \Cref{d:newgeneric}) by computing the \textit{(Backelin) rate}. The notion of rate, first introduced in \cite{Backelin-rate}, measures the growth of the shifts appearing in the minimal graded free resolution of the residue field. An algebra is Koszul if and only if it has minimal rate equal to $1$, and thus having minimal rate (not equal to $1$) can be regarded as a generalization of being Koszul. 
In the main result \cref{general-rate}, we establish that $R$ does not have minimal rate, but $A$ has minimal rate (not equal to 1 when $n\ge 6$).

To aid in our computations throughout this section, we introduce the notion of rate of a series in the following definition, and we define the rate of a standard graded $\kk$-algebra using the rate of the Poincar\'e series of $\kk$.

\begin{definition}
\label{def:series-rate} 
For  a series $\Po(t,u)=\sum_{i,j\geq 0}c_{ij}t^iu^j$ with real coefficients and integer exponents,  we  define
\begin{equation}
   \tau_i=\sup\{j\,|\, j\ge 0 \text{ and } c_{ij}\not=0\}\qquad \text{ and }\qquad \rate(\Po(t,u))=\sup_{i\geq 2}\frac{\tau_i-1}{i-1}.
\end{equation}
In this definition, we use the convention that the supremum of the empty set is $-\infty$ and $\rate(P(t,u))=-\infty$ if $\tau_i=-\infty$ for all $i\ge 2$.

Let $\kk$ be a field and $S$ a standard graded $\kk$-algebra that is not regular. We define
\begin{equation}
\label{rate of series - ring}
\rate(S)=\rate(\Po_\kk^S(t,u)).
\end{equation}
Write $S=Q/I$ with $Q$ a standard graded polynomial ring over $\kk$ and $I\ne 0$ a homogeneous ideal of $Q$ contained in $Q_{\ge 2}$. Let  $m(I)$ denote the maximum degree of a minimal generator of $I$. It follows from \cite[Theorem 2.3.2]{BH93} that 
\begin{equation*}
\rate(S)\ge m(I)-1.
\end{equation*}
When equality holds, we say $S$ has {\it minimal rate}.
\end{definition}
 
Interesting classes of algebras which have minimal rate include quotients by monomial ideals (see \cite{Backelin-rate}), generic toric rings (see \cite{GPW00}), coordinate rings of generic points in projective space under some conditions (see \cite{CDNR01}), and certain compressed Gorenstein algebras (see \cite{BDDR}).  

Next we prove two technical lemmas which allow us to conclude the rate of $A$ and $R$ from the Poincar\'e series formulas in \eqref{e:k-over-A} and \eqref{e:k-over-R}, respectively, in \cref{pseries}.

\begin{lemma}\label{rateprop}
Let $\Qo(z)\in \mathbb R[[z]]$ be a power series with $\Qo(0)\ne 0$. For all $i\ge 1$ let $a_i$, $b_i$ be real numbers such that  $a_1\ne 0$, $b_1=0$, and let $\ell\ge 2$ be an integer. 
Then, the series 
\begin{align}\label{pseries-hyp}
    \Po(t,u)=\frac{\Qo(tu)}{1-tu^\ell\sum_{i\ge 1}(a_i+b_iu)(tu)^i}\,.
\end{align}
 has $\rate(\Po(t,u))=\ell$.
\end{lemma}

\begin{proof}  First we set  
\begin{equation*}
    \Po(t,u)=\sum_{i,j\geq 0}c_{ij}t^iu^j\quad\text{and}\quad \Co_i(u)=\sum_{j\geq 0}c_{ij}u^j,\, \text{for all}\ i\geq 0.
\end{equation*} 
 These yield the following expression for the given series:
\begin{align}\label{pseries-beta}
    \Po(t,u)=\sum_{i\ge 0} \Co_i(u)t^i.
\end{align}
According to  \cref{def:series-rate}, set $\tau_i=\deg(\Co_i(u))$ for all $i\geq 0$. 
We proceed by induction on $i$ to show that for all $i\ge 1$ we have
\begin{equation}
\label{e:c}
\tau_i\le \ell(i-1)+1.
\end{equation}
Write $\Qo(z)=\sum_{i\ge 0}q_iz^i$. By hypothesis, substituting \eqref{pseries-hyp}  into \eqref{pseries-beta},  we have the equality: 
\[
\left(1-tu^{\ell}\sum_{i\ge 1}(a_i+b_iu)(tu)^i\right)\cdot \sum_{i\ge 0}\Co_i(u)t^i=\sum_{i\ge 0}q_i(tu)^i.
\]
Equating the coefficients of $t^i$ on both sides, we obtain for all $i\ge 0$:
\begin{align}
\label{e:equate-coeff}
\Co_i(u) = q_iu^i+\sum_{\substack{p+q=i \\ p\ge 2,\, q\ge 0}}\Co_q(u)u^{\ell+p-1}(a_{p-1}+b_{p-1}u) = q_iu^i+\sum_{p=2}^i\Co_{i-p}(u)u^{\ell+p-1}(a_{p-1}+b_{p-1}u).
\end{align}
 
Let $i\ge 1$. We use the fact that $\tau_i=\deg(\Co_i(u))$.  We first note that $C_0(u)=q_0=\Qo(0)\ne 0$. Further, $C_1(u)=q_1u$, hence $\tau_1\le 1$. Thus \eqref{e:c} holds for $i=1$.   
When $i=2$, since $b_1=0$, we have:
$$
\Co_2(u)=q_2u^2+\Co_0(u)u^{\ell+1}a_1.
$$
Since $a_1\ne 0$, $C_0(u)\ne 0$ and $\ell\ge 2$, we conclude $\tau_2=\ell+1$, which satisfies \eqref{e:c}.

Now let $i\ge 3$. Assume that for all $1\le m<i$ we have the inequalities
\begin{equation}\label{indhyp}
\tau_m\le
    \ell(m-1)+1.
\end{equation}
We want to show \eqref{e:c} holds for $\tau_i$.
We have 
\begin{align*}
\deg\left(\Co_{i-p}(u)u^{\ell+p-1}(a_{p-1}+b_{p-1}u)\right) 
\le & \tau_{i-p} + (\ell+p-1)+1\\
\le & (\ell(i-p-1)+1) + (\ell+p-1)+1\\
= & \ell i-p(\ell-1)+1,
\end{align*}
where the second inequality follows from \eqref{indhyp}. Consequently, 
\begin{align*}
\deg\left(\sum_{p=2}^i \Co_{i-p}(u)u^{\ell+p-1}(a_{p-1}+b_{p-1}u) \right)
\le & \max\{\ell i-p(\ell-1)+1\,|\,2\le p\le i\}\le \ell(i-1)+1. 
\end{align*}
where the second inequality follows since $\ell\ge 2$ and therefore the maximum is attained when $p=2$. Notice also that $i-1<\ell(i-1)$, so $i<\ell(i-1)+1$. Therefore by \eqref{e:equate-coeff} we have
\begin{align*}
\tau_i =\deg\left(\Co_i(u)\right)
= \max\left(i, \deg\left(\sum_{p=2}^i \Co_{i-p}(u)u^{\ell+p-1}(a_{p-1}+b_{p-1}u) \right)\right)
\le  \ell(i-1)+1
\end{align*}
which gives the desired inequality \eqref{e:c}.  Using \cref{def:series-rate}, this implies 
\[
\rate(P(t,u))=\sup_{i\ge 2}\frac{\tau_i-1}{i-1}\leq 
\sup_{i\ge 2}{\displaystyle\frac{\ell(i-1)+1-1}{i-1}=\ell.}
\]
Since $\tau_2=\ell+1$, we also have the reverse inequality and we conclude $\rate(P(t,u))=\ell$. 
\end{proof}

\begin{lemma}\label{rateprop2}
Let $\Qo(z)\in \mathbb R[[z]]$ be a power series with $\Qo(0)\ne 0$. For all $i\ge 1$ let $a_i$ be real numbers with $a_1\ne 0$ and let $\ell\ge 1$ be an integer.  Consider the series 
\begin{align}\label{pseries-hyp2}
    \Po(t,u)=\frac{\Qo(tu)}{1-t^2u^2-t^2u^{\ell+2}\sum_{i\ge 1}a_i(tu)^i}.
\end{align}
Then, $\rate(\Po(t,u))=\frac{\ell}{2}+1$.
\end{lemma}
\begin{proof}
As in the proof of \cref{rateprop}, we set  
\begin{equation*}
    \Po(t,u)=\sum_{i,j\geq 0}c_{ij}t^iu^j\quad\text{and}\quad \Co_i(u)=\sum_{j\geq 0}c_{ij}u^j,\, \text{for all}\ i\ge 0
\end{equation*} 
and $C_i(u)=0$ when $i<0$. In particular, this notation yields the equality $\tau_i=\deg(\Co_i(u))$ for all $i\geq 0$ and the following expression for the given series:
\begin{align}\label{pseries-beta2}
    \Po(t,u)=\sum_{i\ge 0} \Co_i(u)t^i.
\end{align}
Write $\Qo(z)=\sum_{i\ge 0}q_iz^i$. By hypothesis, substituting \eqref{pseries-hyp2} into \eqref{pseries-beta2},  we have the equality: 
\[
\left(1-t^2u^2-t^2u^{\ell+2}\sum_{i\ge 1}a_i(tu)^i\right)\cdot \sum_{i\ge 0}\Co_i(u)t^i=\sum_{i\ge 0}q_i(tu)^i.
\]
Equating the coefficients of $t^i$ on both sides, we obtain the following equalities for each $i\geq 0$:
\begin{align}\label{equate-coeff2}
\Co_i(u)=q_iu^i+\Co_{i-2}(u)u^2+\sum_{\substack{p+m=i \\ p\ge 3,\, m\ge 0}}a_{p-2}\Co_m(u)u^{\ell+p}=q_iu^i+\Co_{i-2}(u)u^2+\sum_{p=3}^ia_{p-2}\Co_{i-p}(u)u^{\ell+p}.
\end{align}

We claim that 
\begin{equation}
\label{eq:tau-claim}
\tau_i\le \left\lfloor\frac{i}{3}\right\rfloor\ell +i,
\end{equation}
which we prove by induction on $i$. As a base case, when $i=0$, equation \eqref{equate-coeff2} implies that $\Co_0(u)=q_0=Q(0)\ne 0$, and thus $\tau_0= 0$. Also, $C_1(u)=q_1u$ and $C_2(u)=q_2u^2+q_0u^2$, which implies that $\tau_1\le 1$ and $\tau_2\le 2$, respectively. Since $C_3(u)=q_3u^3+q_1u^2+a_1q_0u^{\ell+3}$, $a_1\ne 0$, $q_0\ne 0$ and $\ell\ge 1$, we have $\tau_3=\ell+3$. Thus \eqref{eq:tau-claim} holds for $i\le 3$. 

Now let $i\geq 3$ and assume \eqref{eq:tau-claim} holds for all $0\le m<i$. Examining the degrees of the terms of the sum in \eqref{equate-coeff2}, we have for $p\ge 3$: 
\begin{align*}
    \deg\left(a_{p-2}\Co_{i-p}(u)u^{\ell+p}\right)&\le \tau_{i-p}+\ell+p\le \left\lfloor\frac{i-p}{3}\right\rfloor\ell+(i-p)+\ell+p=\left\lfloor\frac{i-p+3}{3}\right\rfloor\ell+i,
\end{align*}
where the induction hypothesis is used to get the second inequality. Hence 
\begin{align*}
    \deg\left(\sum_{p=3}^ia_{p-2}\Co_{i-p}(u)u^{\ell+p}\right)&\le \max\left\{\left\lfloor\frac{i-p+3}{3}\right\rfloor\ell+i\,\Big|\,3\le p\le i\right\}=\left\lfloor\frac{i}{3}\right\rfloor\ell+i,
\end{align*}
where the equality follows from the fact that the maximum is attained when $p=3$. Since $\ell\ge 1$, equation \eqref{equate-coeff2} yields \eqref{eq:tau-claim}, as claimed. Hence when $i\ge 3$, we have  
\begin{align*}
    \frac{\tau_i-1}{i-1}\le \frac{\left\lfloor\frac{i}{3}\right\rfloor\ell}{i-1}+1\le \frac{i}{3(i-1)}\ell+1=\frac{\ell}{3}\left(1+\frac{1}{i-1}\right)+1\le \frac{\ell}{2}+1\,
\end{align*}
and, since $\tau_3=\ell+3$, equalities hold above when $i=3$. Since $\tau_2\le 2$, and hence ${\displaystyle \frac{\tau_2-1}{2-1}}\le 1\le \frac{\ell}{2}+1$, we conclude that $\rate(\Po(t,u))=\frac{\ell}{2}+1$. 
\end{proof}

\begin{theorem}\label{general-rate}
Let $n\geq 2$ be an integer and set $\ell=\bigl\lfloor\frac{n-2}{2}\bigr\rfloor$. Let $\kk$ be an infinite field of characteristic zero or greater than $n$ and $Q=\kk[x_1, \dots, x_n]$. A general quotient $R$ of $Q$ defined by $n+1$ quadrics $f_1, \dots, f_{n+1}$ has the following properties, where $A$ is the Gorenstein ring associated to $R$. 
\begin{enumerate}[\quad$(1)$]
    \item If $n\ge 3$\footnote{If $n=2$, then $A=\kk$ by \cref{small n}, hence $A$ is a regular ring, and rate is not defined.}, then  $\rate(A)=\max\{\ell,1\}$.  In particular, $A$ has minimal rate.
    \item $\rate(R)=\frac{\ell}{2}+1$. 
\end{enumerate} 
\end{theorem}

\begin{proof}
We may assume that $R$ satisfies the properties listed in \cref{d:newgeneric}. In particular, $P$ is a complete intersection, $f_{n+1}$ is a maximal rank element on $P$ and $\reg_P(A)=\ell$.

Recall that by \cref{small n}, $R$ is Koszul when $n=2$ or $n=3$, and $A$ is Koszul when $n=3$, hence they have rate equal to 1.  Since $\ell=0$ when $n\le 3$, both statements hold in this case. 

We assume now $n\ge 4$. By \cref{pseries}, the Poincar\'e series of $\kk$ over $A$ is given by \eqref{e:k-over-A-general}, which matches \eqref{pseries-hyp}, and the Poincar\'e series of $\kk$ over $R$ is given by \eqref{e:k-over-R}, which matches \eqref{pseries-hyp2}.

Below, we will apply \cref{rateprop} and \cref{rateprop2} with $Q(z)=(1-z)^{-n}$,  $a_i=\beta_{i,i+\ell}^P(A)$ and $b_i=0$ for all $i$. To verify the hypotheses of these lemmas, note that $Q(0)\ne 0$. Also, $a_1$ is equal to the minimal number of generators in degree $\ell+1$ of the ideal $G/J$ (with notation as in \cref{not:1}). Using \cref{Hilb R and A}(3), we see that $a_1=C_{\ell+2}\ne 0$. 

(1): If $\ell=1$, then $n=4$ or $n=5$ and $A$ is Koszul by \cref{t:generic}(1); hence $\rate(A)=1$.  If $\ell\ge 2$, then the fact that $\rate(A)=\ell$ follows from equation \eqref{e:k-over-A-general} in \cref{pseries}, \eqref{rate of series - ring}, and \cref{rateprop}. By \cref{Hilb R and A}(3), $G$ has minimal generators in degree $\ell+1$, and since $\reg_P(A)=\ell$, we have $m(G)\le\ell+1$. Thus we have equality $m(G)=\ell+1$ and by \cref{def:series-rate}, $A$ has minimal rate.

(2): Since $n\ge 4$ we have $\ell\ge 1$, and the result follows directly from equation \eqref{e:k-over-R} in \cref{pseries}, \eqref{rate of series - ring}, and \cref{rateprop2}.
\end{proof}

\begin{remark}
Assume the notation and hypotheses in 
\cref{not:1} and let $f_{n+1}$ be a quadratic maximal rank element of $P$. If $n$ is even, then the conclusions of \cref{general-rate} hold, in view of \cref{less}.
\end{remark}

\section{The Yoneda Ext algebra of $A$}
  
In the previous sections, we established that the associated Gorenstein ring $A$ (see \Cref{RPA}) of a general quotient $R$ of $Q$ defined by $n+1$ quadrics satisfies  generalizations of the Koszul property involving Poincar\'e series and rate. In this section we conclude this exploration of generalizations of the Koszul property by studying the Yoneda algebra $\Ext_A(\kk,\kk)$.

It is known that a standard graded $\kk$-algebra $S$ is Koszul if and only if the Yoneda algebra $\Ext_S(\kk,\kk)$ is generated as a $\kk$-algebra by its elements of degree $1$; see \cite[Theorem 1.2]{Lofwall}. In \Cref{t:generic}(1) we proved that the Gorenstein ring $A$ is absolutely Koszul for $n=4$ and $n=5$ but it is not Koszul for $n\geq 6$. In the main result \cref{yoneda} of this section, we show that $\Ext_A(\kk,\kk)$ is generated as a $\kk$-algebra in degrees $1$ and $2$.  Such an algebra is called a $\mathcal K_2$-algebra in \cite{Cassidy-Shelton}. For another class of such algebras see \cite{HS}. 

In general, the Yoneda algebra of $S$ need not be finitely generated as a $\kk$-algebra; see \cite{Lev-Ext}  for a list of known cases in which $\Ext_S(\kk,\kk)$ is finitely generated. In particular, $\Ext_S(\kk,\kk)$ is finitely generated whenever there is a Golod homomorphism from a graded complete intersection onto $S$, and thus \cref{t:generic}(2) yields that the Yoneda algebra of a general quotient $R$ of $Q$ defined by $n+1$ quadrics is also finitely generated. 
 
 Below we collect some known statements that are used in the proof of the main result. While most are originally stated in the context of local rings, they extend as usual to the graded setting. Note that no restrictions on the characteristic of $\kk$ are required in these.

\begin{chunk}\label{vanishing}
Let $\kk$ be a field, $Q=\kk[x_1, \dots, x_n]$ and $S=Q/\fa$, with $\fa$ an ideal generated by homogeneous polynomials of degrees at least $2$. We let $\m$ denote the maximal homogeneous ideal of $S$. If $M$ is a graded $S$-module, we denote by $\nu_S(M)$ the map
\[
\nu_S(M)\colon \Tor^S(\m M, \kk)\to \Tor^S(M, \kk)
\]
induced by the inclusion $\m M\subseteq M$. The vanishing of $\nu_S(M)$ is equivalent to the vanishing of its dual, which is the map $\Ext_S(M, \kk)\to \Ext_S(\m M, \kk)$ induced by the same inclusion. When $M$ is an ideal of $S$, vanishing of $\nu_S(M)$  is also equivalent to the fact that the (dual) maps
\[
\Tor_p^S(S/M,\kk)\to \Tor_p^S(S/\m M,\kk)\qquad\text{and}\qquad  \Ext^p_S(S/M,\kk)\to \Ext^p_S(S/\m M,\kk) 
\]
induced by the canonical projection  $S/\m M\to S/ M$ are zero for all $p>0$. 
\end{chunk}
\begin{chunk}\cite[Theorem 3.2, Proposition 7.5]{Sega-powers}\label{lianathm}
Adopt the notation and hypotheses in \ref{vanishing}.  Then the following assertion holds for the $\kk$-algebra $S$:
\begin{equation}
\label{Koszul-ring} 
    \text{$S$ is Koszul} \iff \nu_S(\m^i)=0\ \text{for all } i \geq 0.
\end{equation}
Moreover, if $S$ is Koszul, then for all finitely generated graded modules $M$ we have:
\begin{equation}\label{reg-powers}
    \reg_S(M)=\inf\{s\ge 0\mid  \nu_S(\m^iM)=0 \text{ for all $i\ge s$}\}\,.
\end{equation}
\end{chunk}

\begin{remark}\label{remark:vanishing}
Adopt the notation and hypotheses in \ref{vanishing}. Let $J$ be a homogeneous ideal generated by a $Q$-regular sequence that is part of a minimal generating set of $\fa$. Set $P=Q/J$, and note that one can regard $\m$ as a $P$-module.  Then for any $i\ge 1$ we have
\begin{equation}\label{ci-to-S}
    \nu_P(\m^i)=0\implies \nu_S(\m^i)=0.
\end{equation}
This statement is a consequence of \cite[Lemma 2]{Lev-Poincare}. This is  explained in \cite[1.12.1]{Maleki-Sega} in the case when $i=1$, and we note that the same argument works for any $i\ge 1$.

Furthermore, the proof of \cite[Theorem 3.15]{Lev} shows that the following assertion holds for the canonical homomorphism $\phi\,\colon S\to S/\m^i$, where $i\ge 1$: 
\begin{equation}\label{powers-Golod}
\nu_S(\m^{i-1})=0\implies\phi\quad\text{is a Golod homomorphism.}
\end{equation}
\end{remark}

\begin{lemma}
\label{combine} 
Adopt the notation and hypotheses in \ref{vanishing}. If $\nu_S(\m^{i-1})=0=\nu_S(\m^i)$ for some $i\ge 2$, then $\Ext_{S/\m^i}(\kk,\kk)$ is generated in degrees $1$ and $2$ if and only if $\Ext_{S}(\kk,\kk)$ is generated in degrees $1$ and $2$. 
\end{lemma}

\begin{proof}
Since $\nu_S(\m^{i-1})=0$, the canonical homomorphism $S\to S/\m^i$ is Golod by \eqref{powers-Golod}.  This, together with the fact that $\nu_S(\m^i)=0$, shows that the hypotheses of \cite[Theorem 5.8]{Lev-Ext} are satisfied, and therefore we have isomorphisms of $\kk$-vector spaces
\[
\frac{\Ext_S^{j}(\kk,\kk)}{\sum_{\substack{a+b=j\\ a,b\ge 1}}\Ext_S^{a}(\kk,\kk)\Ext_S^b(\kk,\kk)}\cong \frac{\Ext_{S/\m^i}^{j}(\kk,\kk)}{\sum_{\substack{a+b=j\\ a,b\ge 1}}\Ext_{S/\m^i}^{a}(\kk,\kk)\Ext_{S/\m^i}^b(\kk,\kk)}
\quad\text{for all $j>2$.}
\]
Consequently, $\Ext_S(\kk,\kk)$ is generated in degrees $1$ and $2$ if and only if $\Ext_{S/\m^i}(\kk,\kk)$ is generated in degrees $1$ and $2$. 
\end{proof}

\begin{proposition}\label{p:ExtS}
Let $Q=\kk[x_1, \dots, x_n]$ with $\kk$ a field.
Let  $J$ and $\fa$  be homogeneous ideals of $Q$ such that $J$ is generated by a quadratic regular sequence and $J\subseteq \fa \subseteq J+(x_1, \dots, x_n)^{\ell+1}$ for some $\ell\ge 2$. Set $P=Q/J$ and $S=Q/\fa$.  If $\reg_P(S)=\ell$, then the Yoneda algebra $\Ext_S(\kk,\kk)$ is generated as a $\kk$-algebra by its elements of degree $1$ and $2$. 
\end{proposition}

\begin{proof}
Let $\n$ denote the maximal ideal of $P$. 
Note that since $J$ is generated by a quadratic regular sequence, we have that $P$ is Koszul (this follows from \cite{tate}; see for example \cite[Remark 6 (3)]{Koszul-survey}). Since $P$ is Koszul, $\Ext_P(\kk,\kk)$ is generated in degree $1$ and \ref{Koszul-ring} implies $\nu_P(\n^i)=0$ for all $i\ge 0$. Hence \cref{combine} gives that $\Ext_{P/\n^{\ell+1}}(\kk,\kk)$ is generated in degrees $1$ and $2$. 

Let $\m$ denote the maximal homogeneous ideal of $S$. 
The hypotheses on $J$ and $\fa$ imply that $P/\n^{\ell+1}\cong S/\m^{\ell+1}$.  Therefore, $\Ext_{S/\m^{\ell+1}}(\kk,\kk)$ is generated in degrees $1$ and $2$.

Since $\reg_P(S)=\ell$ and $P$ is Koszul, it follows from \ref{reg-powers} that $\nu_P(\m^i)=0$ for all $i\ge \ell$. The hypotheses on $J$ and $\fa$  imply that the regular sequence generating $J$ is part of a minimal generating set of $\fa$. Then \ref{ci-to-S} in \cref{remark:vanishing} implies $\nu_S(\m^i)=0$ for all $i\ge \ell$. Using again \cref{combine}, we conclude that  $\Ext_S(\kk,\kk)$ is generated in degrees $1$ and $2$.
\end{proof}

Applying \cref{p:ExtS} in our general setting, we obtain the following.  

\begin{theorem}\label{yoneda}
Let $n\geq 2$ be an integer, $\kk$  an infinite field of characteristic zero or greater than $n$, and $Q=\kk[x_1, \dots, x_n]$.  
A general quotient $R$ of $Q$ defined by $n+1$ quadrics 
$f_1, \dots, f_{n+1}$ 
has the property that $\Ext_A(\kk,\kk)$ is generated as a $\kk$-algebra by its elements of degree $1$ and $2$, where  $A$ is the Gorenstein ring associated to $R$. 
\end{theorem}
\begin{proof}
We may assume that $R$ satisfies the properties listed in \cref{d:newgeneric}. In particular, $P$ is a complete intersection, $f_{n+1}$ is a maximal rank element on $P$ and $\reg_P(A)=\ell$. 
    
If $2\leq n\leq 5$, then $A$ is Koszul by \cref{small n} and \cref{t:generic}(1), so $\Ext_A(\kk,\kk)$ is generated as an algebra by its elements of degree $1$.

Assume $n\geq 6$, so $\ell =\bigl\lfloor\frac{n-2}{2}\bigr\rfloor\geq 2$. 
We observe then that the hypotheses of \cref{p:ExtS} hold with $J=(f_1, \dots, f_n)$, $P=Q/J$, $S=A$ and $\fa=J:(f_{n+1})$.  Indeed, the fact that $\fa\subseteq J+(x_1, \dots, x_n)^{\ell+1}$ follows from \cref{Hilb R and A}(2b).
\end{proof}

\section{Bounds on Betti numbers over $Q$ when $n$ is odd}\label{s:bounds}

In this section, let $\kk$ be an infinite field of characteristic zero or positive characteristic greater than $n$ where $n\ge 2$ is a fixed integer and let $Q=\kk[x_1,\dots,x_n]$ be a standard graded polynomial ring. When $n$ is even, we have precise formulas \eqref{Betti RQ} and \eqref{Betti AQ} for the graded Betti numbers of the rings $R$ and $A$, where $R$ is a general quotient of $Q$ defined by $n+1$ quadrics (as in \Cref{def: generic} and \Cref{d:newgeneric}) and $A$ is the Gorenstein ring associated to $R$. Formulas for Betti numbers of $R$ are not known when $n$ is odd. These Betti numbers are conjectured in \cite[Conjecture 5.8]{MM} to be the minimal ones possible that are consistent with the Hilbert function of $R$. Furthermore, upper bounds for these Betti numbers are given in \cite[Theorem 5.4]{MM}. In \cref{p:overQ}, we assume that $n$ is odd and determine new upper bounds on the graded Betti numbers of $R$ and $A$ in this case, and then we compare them with the ones mentioned above.  In particular, we show that $R$ is level and describe explicitly the degrees of the generators of the defining ideal of $A$.  These results yield exact values for two previously unknown Betti numbers, consistent with \cite[Conjecture 5.8]{MM}.

Recall the rings $\widetilde R$, $\widetilde P$, $\widetilde A$ and $\wtP_{\,[1]}$ from \cref{not:3}: 
\[
\begin{tabular}{ll}
$\widetilde R=Q/(x_1^2, \dots, x_n^2, \widetilde{f}_{n+1})$, & \quad
$\widetilde P=Q/(x_1^2, \dots, x_n^2)$,\\
$\widetilde A=Q/((x_1^2, \dots, x_{n}^2)\colon \widetilde{f}_{n+1})$, & \quad
$\wtP_{\,[1]}=Q/(x_1^2)$.
\end{tabular} 
\]
where $\widetilde{f}_{n+1}=(x_1+\dots+x_n)^2$. 
Also, we recall that $\ell\coloneq \left\lfloor\frac{n-2}{2}\right\rfloor.$

\begin{lemma}
\label{generic Betti P1}    
If $n\geq 3$ is odd, then a general quotient $R$ of $Q$ defined by $n+1$ quadrics $f_1, \dots, f_{n+1}$ satisfies the following properties for all $i,j\geq 0$, where $A$ is the Gorenstein ring associated to $R$ and $\Pu{1}=Q/(f_1)$:
\begin{enumerate}[\quad$(1)$]
\item $\beta_{i,j}^{\Pu{1}}(R)=\beta_{i,j}^{\wtP_{\,[1]}}(\widetilde R)$.\label{eq:R}
\item $\beta_{i,j}^{\Pu{1}}(A)=\beta_{i,j}^{\wtP_{\,[1]}}(\widetilde A)$. \label{eq:A}
\end{enumerate}
\end{lemma}

\begin{proof} 
We may assume that $R$ satisfies the properties listed in \cref{d:newgeneric}. In particular, $P$ is a complete intersection, $f_{n+1}$ is a maximal rank element on $P$ and $\reg_{\Pu{1}}(R)=\ell+1$.
As in \cref{rem: general - hilb}, we have equalities $\fh_A(u)=\fh_{\wtA}(u)$ and  $\fh_R(u)=\fh_{\wtR}(u)$, and we also have $\fh_{\Pu{1}}(u)=\fh_{\wtP_{\,[1]}}(u)$. 
Now in view of \cref{Betti-strand2}, it suffices to prove \eqref{eq:R} when  $j-i\ne \ell+1$ and \eqref{eq:A} when  $j-i\ne \ell$.  

By \cref{p:reg ex}, we have 
\begin{equation}
\label{reg-new}
\reg_{\wtP_{\,[1]}}(\wtR)=\ell+1=\reg_{\Pu{1}}(R).
\end{equation}
It follows that $\beta_{i,j}^{\Pu{1}}(R)=0=\beta_{i,j}^{\wtP_{\,[1]}}(\widetilde R)$ when $j-i>\ell+1$. 
By \cref{ex: wlp} we have that $\widetilde{f}_{n+1}$ is a maximal rank element on  $\widetilde{P}$.  Thus applying \cref{l: A and R over Pu}(2) to $P$ and $\widetilde{P}$ gives $\beta_{i,j}^{\Pu{1}}(R)=\beta_{i,j}^{\wtP_{\,[1]}}(\widetilde R)$ when $j-i\le \ell$. Thus, \cref{Betti-strand2} gives \eqref{eq:R} for all $i,j\ge 0$. 

By \cref{l: A and R over Pu}(1), we have that \eqref{eq:A} holds when $j-i<\ell$. We need to show \eqref{eq:A} holds when $j-i>\ell$.  When $j-i>\ell$,  the following equality follows from \eqref{Betti AQ} and \cref{Betti reduction}:
\begin{align*}
    \beta_{i,j}^{\wtP_{\,[1]}}(\widetilde A)&=
       \begin{cases} 
       \binom{n-1}{i}, &\text{ if } j = 2i-2 \\    
       0,&\text{ otherwise. }
    \end{cases}
\end{align*}

Using \eqref{reg-new} and \cref{p:Betti-over-P}(3c), we conclude that $\beta_{i,j}^{\Pu{1}}(A)=\beta_{i,j}^{\wtP_{\,[1]}}(\widetilde A)$ when $j-i>\ell$, finishing the proof. 
\end{proof}

Now we proceed to give upper bounds on the Betti numbers of $R$ and $A$ over $Q$ when $n$ is odd. Our bounds are given in terms of the Betti numbers of the rings $\wtR$ and $\wtA$; these Betti numbers have been determined in \cite{proceedings}, as recalled below.

\chunk
\label{BettiRA}
For the case of an odd number of variables $n\geq 3$ in the polynomial ring $Q$, the authors of the current paper described the graded Betti numbers of $\wtR$ and of $\wtA$ as $Q$-modules, in \cite[Theorems 4.11 and 4.13]{proceedings}, respectively. 
With the sequence $\{\rho_i(n-1)\}$ as defined in \cref{not: rho} and setting $\rho_i=\rho_i(n-1)$ for all $i$, the Betti table of $\wtR$ over $Q$ is as follows: 

\smallskip
\begin{center}
\setlength{\tabcolsep}{3pt} 
\begin{tabular}{r| c c c c c c c c c c c c c c }
\small{$\beta_{i,j}^Q(\wtR)$} & \small{$0$} & \small{$1$} & \small{$2$} & \small{$3$}  & $\ldots$ & \small{$\ell$} & \small{$\ell+1$} & \small{$\ell+2$} & \small{$\ell+3$} & \small{$\ell+4$} & $\ldots$ & \small{$n-1$} & \small{$n$} \\
\hline
$0$ & $\binom{n+1}{0}$ & - & - & - & $\cdots$  & - & - & - & 
- & - & $\cdots$ & - &-\\ 
$1$ & -  & $\binom{n+1}{1}$ & - & - &$\cdots$ & - & - & - & - & - & $\cdots$ & -  & - \\ 
\vphantom{\huge{A}} \normalsize{$\vdots$}&$\vdots$&$\vdots$&$\vdots$&$\vdots$&$\ddots$&$\vdots$&$\vdots$&$\vdots$&$\vdots$ & $\vdots$&$\vdots$&$\vdots$&$\vdots$ \\ 
\vphantom{\huge{A}} \normalsize{$\ell$} &  - & - & - & - & - &$\binom{n+1}{\ell}$ &- &- & - & - & $\cdots$ & - & - \\
\vphantom{\huge{A}} \normalsize{$\ell+1$} & - & - & $\rho_{2}$ &$\rho_3$ &$\cdots$ & $\rho_{_{\ell}}$ & $\rho_{_{\ell+1}}+\binom{n}{
\ell}$ & $\rho_{_{\ell+2}}$ & $\rho_{_{\ell+3}}$ & $\rho_{_{\ell}}$ &$\ldots$ & $\rho_{2}$ & - \\
$\ell+2$& - & - & - & $\rho_2$ &$\cdots$ &$\rho_{_{\ell-1}}$& $\rho_{_{\ell}}$ & $\rho_{_{\ell+1}}$ & $\rho_{_{\ell+2}}$ & $\rho_{_{\ell+3}}$ & $\cdots$ &  $\rho_{3}$ &$\rho_{2}$ \\
\end{tabular}
\end{center}
\smallskip

With the sequence $\{\gamma_i(n-1)\}$ as defined in \cref{not: gamma} and setting $\gamma_i=\gamma_i(n-1)$ for all $i$, the Betti table of $\wtA$ over $Q$ is as follows: 

\smallskip
\begin{center}
\setlength{\tabcolsep}{3.5pt} 
\begin{tabular}{r| c c c  c c c c c c c c c c c c c}
\footnotesize{$\beta_{i,j}^Q(\wtA)$} & \footnotesize$0$ & \footnotesize$1$  & \footnotesize$2$ &$\cdots$  & \footnotesize$\ell-1$ & \footnotesize$\ell$ & \footnotesize$\ell+1$ &\footnotesize$\ell+2$ &\footnotesize$\ell+3$ &\footnotesize$\ell+4$ & $\cdots$ &\footnotesize$n-2$ & \footnotesize$n-1$& \footnotesize$n$ \\
\hline
\footnotesize\footnotesize\footnotesize$0$ & \footnotesize$\binom{n}{0}$ & - & - & $\cdots$ & - & - & - & - &-&-&$\cdots$ & - & - &- \\ 
\footnotesize\footnotesize$1$ & -  &\footnotesize $\binom{n}{1}$ &-   &$\cdots$ &- &- &- &- &- &-&$\cdots$ &- &- &-\\ 
\vphantom{\huge{I}}\footnotesize$\vdots$&$\vdots$&$\vdots$&$\vdots$&$\ddots$&$\vdots$&$\vdots$&$\vdots$&$\vdots$&$\vdots$&\vdots &$\ddots$&$\vdots$&$\vdots$&$\vdots$\\ 
\vphantom{\huge{I}}\footnotesize$\ell-1$ & - &-  &- & $\cdots$ &\footnotesize$\binom{n}{\ell-1}$ &- &- &-&-&-&$\cdots$ & - & - &-\\
\footnotesize$\ell$& - & $\gamma_{_1}$  &$\gamma_{_2}$&$\cdots$  & $\gamma_{_{\ell-1}}$ & $\gamma_{_{\ell}}$\footnotesize$+\binom{n-1}{\ell-1}$ & $\gamma_{_{\ell+1}}$ &$\gamma_{_{\ell}}$ &$\gamma_{_{\ell-1}}$& $\gamma_{_{\ell-2}}$& $\cdots$ &$\gamma_{_{1}}$ & -&-\\
\vphantom{\huge{I}}\footnotesize$\ell+1$& - &- & $\gamma_{_1}$ &$\cdots$  & $\gamma_{_{\ell-2}}$ & $\gamma_{_{\ell-1}}$ & $\gamma_{_{\ell}}$ &$\gamma_{_{\ell+1}}$ &$\gamma_{_{\ell}}$\footnotesize$+\binom{n-1}{\ell-1}$ &$\gamma_{_{\ell-1}}$& $\cdots$ &$\gamma_{_{2}}$ &$\gamma_{_{1}}$&- \\
\vphantom{\huge{I}}\footnotesize$\ell+2$ &- & - & -&$\cdots$ & - & - &- &-&-&\footnotesize$\binom{n}{\ell+4}$&$\cdots$ & -& -&- \\ 
\vphantom{\huge{I}} \footnotesize\vdots&\vdots &\vdots&\vdots&$\ddots$&\vdots&\vdots&\vdots&\vdots&\vdots&$\vdots$&$\ddots$&\vdots&\vdots&\vdots\\ 
\footnotesize$n-3$ & -   & -  &- &$\cdots$ & - & - & - & - & -&- &$\cdots$ &-&\footnotesize$\binom{n}{n-1}$ & - \\
\footnotesize$n-2$ & -   & - &- &$\cdots$ & - & - & - & - & - &-&$\cdots$ & - &-&\footnotesize$\binom{n}{n}$ \\
\end{tabular}
\end{center}
\smallskip

In view of \cite[Proposition 4.10]{proceedings} and \eqref{e:symm},  the following equalities hold: 
\begin{align}
\beta_{i,j}^Q(\wtR)&=\beta_{i,j}^{\wtP_{\,[1]}}(\wtR)+\beta_{i-1,j-2}^{\wtP_{\,[1]}}(\wtR); \label{t:RAoverQ}\\ 
\beta_{i,j}^Q(\wtA)&=\beta_{i,j}^{\wtP_{\,[1]}}(\wtA)+\beta_{i-1,j-2}^{\wtP_{\,[1]}}(\wtA) 
\label{t:RAoverQ2}.
\end{align} 

\begin{theorem}
\label{p:overQ}
Let $n\geq 3$ be an odd integer, $\kk$ an infinite field of characteristic zero or greater than $n$ and $Q=\kk[x_1, \dots, x_n]$. A general quotient $R$ of $Q$ defined by $n+1$ quadrics $f_1, \dots, f_{n+1}$ satisfies the following, where $A$ is the Gorenstein ring associated to $R$:
\[
\beta_{i,j}^Q(R)\le \beta_{i,j}^Q(\widetilde R)\qquad\text{and}\qquad 
\beta_{i,j}^Q(A)\le \beta_{i,j}^Q(\widetilde A)\, 
\quad\text{for all $i,j\geq 0$,}
\]
where $\beta_{i,j}^Q(\widetilde R)$ and $\beta_{i,j}^Q(\widetilde A)$ are described in \ref{BettiRA}.
\end{theorem}

\begin{proof}
We prove the first inequality, noting that the proof of the second one is similar.

By \cite{eisen}, we have an exact sequence
\begin{equation}
\label{exact-sequence}
\dots\to  \Tor_{i-1}^{\Pu{1}}(R,\kk)_{j-2}\to \Tor_i^{Q}(R,\kk)_j\to \Tor_i^{\Pu{1}}(R,\kk)_j\to \Tor_{i-2}^{\Pu{1}}(R,\kk)_{j-2}\to \dots
\end{equation}
Thus we have the following:
\begin{align*}
\beta_{i,j}^Q(R)&\le \beta_{i,j}^{\Pu{1}}(R)+\beta_{i-1, j-2}^{\Pu{1}}(R)=\beta_{i,j}^{\Pu{1}}(\widetilde R)+\beta_{i-1, j-2}^{\Pu{1}}(\widetilde R)=\beta_{i,j}^Q(\widetilde R),
\end{align*} 
where the inequality follows from \eqref{exact-sequence}, the first equality follows from \cref{generic Betti P1}, and the second equality follows from \eqref{t:RAoverQ}.
\end{proof}

\begin{remark}
\label{r:n=3}
When $n=3$, the inequalities in \cref{p:overQ} are in fact equalities.  Indeed, the Betti tables of $R$ and $A$ over $Q$ are given below: 
\smallskip
\[
\begin{tabular}[t]{r| c c c c }
$\beta_{i,j}^Q(R)$ & $0$ & $1$ & $2$ & $3$  \\
\hline
$0$ & $1$ & - & - & -\\ 
$1$ & -  & $4$ & $2$ & - \\ 
$2$ & - & -  & $3$ &$2$ \\ 
\end{tabular}
\qquad\quad
\begin{tabular}[t]{r| c c c c }
$\beta_{i,j}^Q(A)$ & $0$ & $1$ & $2$ & $3$ \\
\hline
$0$ & $1$ & $2$ & $1$ & -\\ 
$1$ & -  &$1$ &$2$  &$1$ \\
\end{tabular}
\]
\smallskip
and these match the Betti tables of $\widetilde R$ and $\widetilde A$, as can be seen from \cite[Example 4.12, 4.14]{proceedings}. To explain the Betti table of $R$, use \cite[Corollary 4.4]{MM}. To explain the Betti table of $A$, observe that $\fh_A= (1,1)$ by \cref{Hilb R and A},  and hence $A\cong\wtA\cong \kk[x,y,z]/(x^2,y,z)$.  However, when $n\ge 5$ we do not expect equalities to hold in \cref{p:overQ} for all $i$, $j$. 
\end{remark}

In contrast to the case when $n$ is even,
only partial information is known about the Betti numbers of $R$ as in \cref{p:overQ} when $n\geq 5$ is odd; see \cite[Theorem 5.4]{MM}. It is known that $\reg(R)=\ell+2$ and the Betti table of $R$ coincides in the first $\ell$ strands with the Betti table of the Koszul complex on $n+1$ quadrics. The same holds for $\widetilde R$ and in particular, equalities hold in \cref{p:overQ} when $j-i\le \ell$ or $j-i\ge \ell+2$.  On the other hand, the $(\ell+1)$th and $(\ell+2)$th strands of the Betti table of $R$ could potentially have \textit{ghost terms} between them; these are graded free modules that appear as summands in two consecutive terms of the minimal free resolution. In \cite[Conjecture 5.8]{MM}, it is conjectured that there are no ghost terms in these two strands, as is the case when $n=3$. We show in \cref{c:RoverQ-known} that ghost terms cannot occur in the beginning of the resolution of $R$ or at the end, by showing that only one of the entries in the $(\ell+1)$th or $(\ell+2)$th  strand can be nonzero in homological degrees $2$ and $n$.

The next two results collect the known Betti numbers of $R$ and $A$, respectively, together with the new ones that can be deduced from \cref{p:overQ}, and display the resulting Betti tables, where $-$ indicates $0$ and $*$ indicates an undetermined value. 

\begin{corollary}\label{c:RoverQ-known}
With the hypotheses of \cref{p:overQ} when $n\ge 5$, the Betti numbers of $R$ as a $Q$-module satisfy the following:
\begin{enumerate}[\quad$(1)$]
\item $\beta_{i,j}^Q(R)=\binom{n+1}{i}$ when $j=2i$ and $0\leq i \leq\ell$;
\item $\beta_{i,j}^Q(R) = 0$ when $i,j$ satisfy any of the following conditions:
    \begin{enumerate}[\quad$(a)$]
        \item $j-i\le \ell$ and $j \neq 2i$;
        \item $j-i>\ell+2$;
        \item $j-i=\ell+1$ and $i=n$;
        \item $j-i=\ell+2$ and $i=2$.
    \end{enumerate}
\item $\beta_{2, \ell+3}^Q(R)=\begin{cases}
C_{\ell+2} &\mbox{if } n\ge 7\\
C_{3}+ \binom{6}{2}&\mbox{if } n=5
\end{cases}$
\item $\beta_{n,n+\ell+2}^Q(R)=C_{\ell+2}$
\end{enumerate}

In particular,  $R$ is level, with socle in degree $n-\ell-1$, and
$\dim_\kk(\Soc(R))=C_{\ell+2}$. When $n\geq 7$, the Betti table of $R$ as a $Q$-module has the following shape:
\begin{center}
\setlength{\tabcolsep}{3pt} 
\begin{tabular}{r| c c c c c c c c c c c c  }
\small{$\beta_{i,j}^Q(R)$} & \small{$0$} & \small{$1$} & \small{$2$} & \small{$3$}  & $\ldots$ & \small{$\ell$} & \small{$\ell+1$} & \small{$\ell+2$} & $\ldots$ & \small{$n-1$} & \small{$n$} \\
\hline
$0$ & $\binom{n+1}{0}$ & - & - & - & $\cdots$  & - & - & - & $\cdots$ & - &-\\ 
$1$ & -  & $\binom{n+1}{1}$ & - & - &$\cdots$ & - & - & - & $\cdots$ & -  & - \\ 
\vphantom{\huge{A}} \normalsize{$\vdots$} &$\vdots$&$\vdots$&$\vdots$&$\vdots$&$\ddots$&$\vdots$&$\vdots$&$\vdots$&$\vdots$&$\vdots$&$\vdots$ \\ 
\vphantom{\huge{A}} \normalsize{$\ell$} &  - & - & - & - & - &$\binom{n+1}{\ell}$ &- &- & $\cdots$ & - & - \\
\vphantom{\huge{A}} \normalsize{$\ell+1$} & - & - & $C_{\ell+2}$ &$*$&$\cdots$ & $*$ & $*$& $*$  &$\cdots$ & $*$& - \\
$\ell+2$& - & - & - &$*$&$\cdots$ &$*$& $*$& $*$   & $\cdots$& $*$&$C_{\ell+2}$
\end{tabular}
\end{center}
\end{corollary}

\begin{proof}
Since $n$ is odd, we have $n=2\ell+3$. The Betti numbers in (1), (2a), (2b), (3), and (4)  come from  \cite[Proposition 3.8]{proceedings}. In view of \cref{p:overQ}, the Betti numbers in (2c) and (2d) come from the fact that $\beta_{n, n+\ell+1}^Q(\widetilde R)=0$ and $\beta_{2, \ell+4}^Q(\widetilde R)=0$, respectively, which can be observed from the first table in  \Cref{BettiRA}.
\end{proof}
A similar result is given for the ring $A$. 

\begin{corollary}\label{c:AoverQ-known}
With the hypotheses of \cref{p:overQ} when $n\ge 5$, the Betti numbers of $A$ as a $Q$-module satisfy the following:
\begin{enumerate}[\quad$(1)$]
\item $\beta_{i,j}^Q(A)=\binom{n}{i}$ when $i,j$ satisfy any of the following conditions:
\begin{enumerate}[\quad$(a)$]
    \item $j=2i$ and $0\leq i \leq\ell-1$;
    \item $j=2i-2$ and $\ell +4\leq i\leq n$.
\end{enumerate}
\item $\beta_{i,j}^Q(A) = 0$ when $i,j$ satisfy any of the following conditions:
    \begin{enumerate}[\quad$(a)$]
        \item $j-i\le \ell-1$ and $j \neq 2i$;
        \item $j-i>n-2$;
        \item $j-i=\ell$ and $i=n-1$;
        \item $j-i=\ell+1$ and $i=1$.
    \end{enumerate}
\item $\beta_{1, \ell+1}^Q(A)=\begin{cases}
C_{\ell+2} &\mbox{if } n\ge 7\\
5+C_{3}&\mbox{if } n=5
\end{cases}$
\item $\beta_{n-1,n+\ell}^Q(A)=C_{\ell+2}$.
\end{enumerate}
In particular, the defining ideal of $A$ as a quotient of $Q$ is generated in degrees $2$ and $\ell+1$ and its minimal number of generators is $n+C_{\ell+2}$. When $n\geq 7$, the Betti table of $A$ as a $Q$-module has the following shape: 
\begin{center}
\setlength{\tabcolsep}{3.5pt} 
\begin{tabular}{r| c c c  c c c c c c c c c c c c c}
\footnotesize{$\beta_{i,j}^Q(A)$} & \footnotesize$0$ & \footnotesize$1$  & \footnotesize$2$ &$\cdots$  & \footnotesize$\ell-1$ & \footnotesize$\ell$ & \footnotesize$\ell+1$ &\footnotesize$\ell+2$ &\footnotesize$\ell+3$ &\footnotesize$\ell+4$ & $\cdots$ &\footnotesize$n-2$ & \footnotesize$n-1$& \footnotesize$n$ \\
\hline
\footnotesize\footnotesize\footnotesize$0$ & \footnotesize$\binom{n}{0}$ & - & - & $\cdots$ & - & - & - & - &-&-&$\cdots$ & - & - &- \\ 
\footnotesize\footnotesize$1$ & -  &\footnotesize $\binom{n}{1}$ &-   &$\cdots$ &- &- &- &- &- &-&$\cdots$ &- &- &-\\ 
\vphantom{\huge{I}}\footnotesize $\vdots$&$\vdots$&$\vdots$&$\vdots$&$\ddots$&$\vdots$&$\vdots$&$\vdots$&$\vdots$&$\vdots$&$\vdots$ &$\ddots$&$\vdots$&$\vdots$&$\vdots$\\  
\vphantom{\huge{I}}\footnotesize$\ell-1$ & - &-  &- & $\cdots$ &\footnotesize$\binom{n}{\ell-1}$ &- &- &-&-&-&$\cdots$ & - & - &-\\
\footnotesize$\ell$& - & $C_{\ell+2}$  &$*$&$\cdots$  & $*$ & $*$ & $*$ &$*$ &$*$& $*$& $\cdots$ &$*$ & -&-\\
\vphantom{\huge{I}}\footnotesize$\ell+1$& - &- & $*$ &$*$  & $*$ & $*$ & $*$ &$*$ &$*$ &$*$& $\cdots$ &$*$ &$C_{\ell+2}$&- \\
\vphantom{\huge{I}}\footnotesize$\ell+2$ &- & - & -&$\cdots$ & - & - &- &-&-&\footnotesize$\binom{n}{\ell+4}$&$\cdots$ & -& -&- \\ 
\vphantom{\huge{I}} \footnotesize$\vdots$&$\vdots$ &$\vdots$&$\vdots$&$\ddots$&$\vdots$&$\vdots$&$\vdots$&$\vdots$&$\vdots$&$\vdots$&$\ddots$&$\vdots$&$\vdots$&$\vdots$\\ 
\footnotesize$n-3$ & -   & -  &- &$\cdots$ & - & - & - & - & -&- &$\cdots$ &-&\footnotesize$\binom{n}{n-1}$ & - \\
\footnotesize$n-2$ & -   & - &- &$\cdots$ & - & - & - & - & - &-&$\cdots$ & - &-&\footnotesize$\binom{n}{n}$ \\
\end{tabular}
\end{center}
\end{corollary}

\begin{proof}
Since $n$ is odd, we have $n=2\ell+3$. The Betti numbers in (1), (2a), (2b), (3), and (4)  come from  \cite[Proposition 3.8]{proceedings}.  In view of \cref{p:overQ}, the Betti numbers in (2c) and (2d) come from the fact that $\beta_{n-1, n-1+\ell}^Q(\widetilde A)=0$ and $\beta_{1, \ell+2}^Q(\widetilde A)=0$, respectively, which can be observed from the second table in \eqref{BettiRA}. 
\end{proof}

While we do not know the exact values of the entries marked $*$ in the Betti tables above, we note that some additional information on existing symmetries comes from the fact that $A$ is Gorenstein, and also from duality for $R$, as pointed out in \cite[Proposition 3.8(4)]{proceedings}. 

\begin{ex}
For $n=7$, we record below a comparison of the bounds on the Betti numbers of $R$ from \cite[Theorem 5.4]{MM} (on the left) with the bounds given by \cref{p:overQ} and \eqref{e:small rho} (on the right). 
\smallskip
\[
\footnotesize{
\setlength{\arraycolsep}{3.9pt}
\begin{array}{r|c c c c c c c c}
 & 0 & 1 & 2 & 3 & 4 & 5 & 6 & 7 \\ 
  \hline 
0 & 1 & - & - & - & - & - & - & - \\ 
1 & - & 8 & - & - & - & - & - & - \\ 
2 & - & - & 28 & - & - & - & - & - \\ 
3 & - & - & 14 & \leq 161 & \leq 196 & \leq 175 & \leq 78 & \leq 14 \\ 
4 & - & - & \leq 14 & \leq 78 & \leq 245 & \leq 196 & \leq 105 & 14 \\ 
\end{array}
\quad\,\,
\begin{array}{r|c c c c c c c c}
  & 0 & 1 & 2 & 3 & 4 & 5 & 6 & 7 \\ 
  \hline 
0 & 1 & - & - & - & - & - & - & - \\ 
1 & - & 8 & - & - & - & - & - & - \\ 
2 & - & - & 28 & - & - & - & - & - \\ 
3 & - & - & 14 & \leq 126 & \leq 132 & \leq 70 & \leq14 & 0 \\ 
4 & - & - & 0 & \leq 14 & \leq 105 & \leq 132 & \leq 70 & 14 \\ 
\end{array}
}
\]
\end{ex}

\section{Appendix}
In this section, let $\kk$ be a field and  $P=\oplus_{i\geq 0}P_i$ and $R=\oplus_{i\geq 0}R_i$  standard graded $\kk$-algebras with $P_0 =\kk=R_0$. Let  $\p\coloneq P_{\geq 1}$ and $\m\coloneq R_{\geq 1}$ be the homogeneous maximal ideals of
$P$ and $R$ respectively. Consider a surjective homomorphism of standard graded $\kk$-algebras
\[
\varkappa\colon (P,\p,\kk)\to (R,\m,\kk).
\]
The main results of this appendix, \cref{l:Massey-lemma} and  \cref{p:Golod-I}, give general criteria for the map $\varkappa$ to be Golod, which are key ingredients of the results in \cref{sec:Poincare}. 

\begin{chunk}
\label{Golod-hom}
Let $(\cD,\partial^{\cD})$ be a minimal free resolution of $\kk$ over $P$ with a graded-commutative  DG-algebra structure; such a resolution always exists; see \cite{Gu68}, \cite{Sch}. Let $(\cA,\partial^\cA)$ be the DG-algebra given by $\cA= R\otimes_P \cD$ and $\partial^\cA=1\otimes_P\partial^\cD$. If $x\in\cA$ is a homogeneous element,  let $|x|$ denote the homological degree of $x$ and set $\ov x=(-1)^{|x|+1}x$. 

 Let $\bsh=\{h_\lambda\}_{\lambda\in \Lambda}$ be a homogeneous basis of the bigraded $\kk$-vector space $\HH_{\ge 1}(\cA)_*$. Each basis element $h_\lambda$,  can be written as  
\[h_\lambda=[z_\lambda],\quad\text{where}\quad z_\lambda\in\cA_{i_\lambda,j_\lambda}\quad \text{is a cycle}.\] 
Thus, $i_\lambda$ denotes the homological degree and $j_\lambda$ denotes the internal degree of $h_\lambda$. 

 If  $h\in \bigsqcup_{n=1}^\infty\bsh^n$ we let $|h|$ denote the integer $n\ge 1$ with $h\in \bsh^n$. For $h=(h_{\lambda_1}, \dots, h_{\lambda_n})\in \bsh^n$ and each  $1\leq k\leq n-1$ we set 
\[
h_{\downarrow k}=(h_{\lambda_1}, \dots, h_{\lambda_k})\in \bsh^k\quad\text{and}\quad h_{\uparrow k}=(h_{\lambda_{k+1}}, \dots, h_{\lambda_n})\in \bsh^{n-k}.
\]

According to Gulliksen \cite{Gu68}, $\varkappa$ is said to be a \emph{Golod homomorphism} if there exists a function
$\mu\colon\bigsqcup_{n=1}^{\infty}\bsh^n\to\cA$, called {\it{trivial Massey operation}},  satisfying:
\begin{align}
\label{eq:tmo1}
\mu(h_\lambda)&
\text{ is a cycle in the homology class of $h_\lambda$
for each } h_\lambda\in\bsh;
   \\
\label{eq:tmo2}
 \dd^\cA\mu(h)&=\sum_{k=1}^{|h|-1}\ov{\mu(h_{\downarrow k})}\mu(h_{\uparrow k}) \quad\text{for each $h\in \textstyle{\bigsqcup_{n=2}^\infty\bsh^n}$\,;}\\
   \label{eq:tmo3}
\mu(\bsh^n)&\subseteq \m\cA\quad\text{for each }n\geq 1\,.
\end{align}
 
If $\mu$  is a trivial Massey operation, then $\mu(h)\in\cA_{i(h),j(h)}$, where  
\begin{equation}
\label{e:ijh}
i(h)=\sum_{k=1}^n i_{\lambda_k}+(n-1)\qquad \text{and}\qquad j(h)=\sum_{k=1}^n j_{\lambda_k}. 
\end{equation}
\end{chunk}
The next result gives a general criterion for constructing a Massey operation on $\mathcal A$. Below, given positive integers $c$, $i$, we consider the map 
\[
\psi_i^c\,\colon\, \HH_i(\cA/\m^{2c}\cA)\to \HH_i(\cA/\m^c\cA)
\]
induced in homology by the canonical projection. 
\begin{lemma}
\label{l:Massey-lemma}
Adopt the notation and hypotheses of \ref{Golod-hom}.  Assume that there exists an integer $c>0$ such that 
\begin{equation}
\label{map0}
\psi_i^c=0\qquad \text{for all $i\ge 1$}, 
\end{equation}
and consider homogeneous cycles $z_\lambda\in \m^c\mathcal A_{i_\lambda, j_\lambda}$ with $h_\lambda=[z_\lambda]$ for all $\lambda\in \Lambda$. 
If for each $h\in \bigsqcup_{n=2}^\infty\bsh^n$ one of the following conditions holds
\begin{enumerate}[\quad $(1)$]
    \item $\HH_{i(h)-1}(\mathcal A)_{j(h)}=0$; or
    \item $z_{\lambda_i}z_{\lambda_j}=0$ for all $i,j\in [|h|]$ with $i\ne j$,
\end{enumerate}
then $\mathcal A$ admits a trivial Massey operation $\mu$ with  $\mu(\mathcal A)\subseteq \m^c\mathcal A$. 
\end{lemma}
\begin{proof}
Choose $\bsh=\{h_\lambda\}_{\lambda\in \Lambda}$ a homogeneous basis of the bigraded $\kk$-vector space $\HH_{\ge 1}(\cA)_*$.
The existence of homogeneous cycles $z_\lambda\in \m^c\mathcal A_{i_\lambda, j_\lambda}$ with $h_\lambda=[z_\lambda]$ for all $\lambda\in \Lambda$ follows by \eqref{map0}. 
We construct a trivial Massey operation on $\mathcal A$ by induction. Namely, for $n\ge 1$ and each $h\in\bsh^n$ (meaning $|h|=n$), we construct $\mu(h)$ satisfying conditions \eqref{eq:tmo1}, \eqref{eq:tmo2}, and also the conditions
\begin{gather}
\label{e:ijnew}
\mu(h)\in \m^c\cA_{i(h),j(h)},\\
\label{mu0}
\text{if $|h|\ge 2$ and $h$ satisfies (2), then $\mu(h)=0$.}
\end{gather}
Assume $n=1$. Then, for each $\lambda\in \Lambda$, set $\mu(h_\lambda)=z_\lambda$, where $z_\lambda$ is the cycle chosen above. Clearly, \eqref{eq:tmo1} and \eqref{e:ijnew} hold, and \eqref{mu0} holds vacuously. 

Assume $n\geq 2$ and that $\mu(h)$ has been defined for all  $h\in \bsh^k$ with $1\le k<n$, so that \eqref{eq:tmo1}, \eqref{eq:tmo2}, \eqref{e:ijnew} and \eqref{mu0} hold. Let $h=(h_{\lambda_1}, \dots, h_{\lambda_n})\in \bsh^n$ and set
\begin{equation}
    \label{e:claimnew}
   \omega\colon=\sum_{k=1}^{n-1}\ov{\mu(h_{\downarrow k})}\mu(h_{\uparrow k})\,.
\end{equation}
Using the Leibniz rule, the induction hypothesis, and \eqref{e:ijh} one can check that  \[\omega\in\m^{2c} \mathcal A_{i(h)-1, j(h)}\quad\text{and}\quad \partial^\mathcal A(\omega)=0.\] 

Assume $h$ satisfies (1). The condition $\HH_{i(h)-1}(\mathcal A)_{j(h)}=0$ implies that  $\omega$ is a boundary. Choose  $u\in \mathcal A_{i(h),j(h)}$ with $\partial^\mathcal A(u)=\omega$.  By \eqref{map0}, there exist $v\in \mathcal A_{i(h)+1,j(h)}$ and $u'\in \m^c\mathcal A_{i(h),j(h)}$ such that $u=\partial^{\mathcal A}(v)+u'$, and hence $\partial^{\mathcal A}(u)=\partial^{\mathcal A}(u')$. If we set $\mu(h)=u'$, then \eqref{eq:tmo2} and \eqref{mu0} hold. 

Assume $h$ satisfies (2). If $n=2$, we have  $\omega=z_{\lambda_1}z_{\lambda_2}=0$, and we set $\mu(h)=0$.  If $n>2$, then both $h_{\downarrow k}$  and $h_{\uparrow k}$ satisfy condition (2) for all $1\le k\le n-1$.  Observe that $|h_{\downarrow k}|=k$ and $|h_{\uparrow k}|=n-k$.  By the induction hypothesis applied to $h_{\downarrow k}$  and $h_{\uparrow k}$, we have
\begin{align*}
\mu(h_{\downarrow k})&=0\quad\text{ when } 2\leq k\leq n-1,\text{and}\\
\mu(h_{\uparrow k})&=0\quad\text{ when } 1\leq k\leq n-2.
\end{align*}
As a consequence, we have $\mu(h_{\downarrow k})\mu(h_{\uparrow k})=0$ for all $1\le k\le n-1$, hence $\omega=0$, and we may define $\mu(h)=0$. 
\end{proof}

In the following result, we specialize the previous lemma to fit our purposes. 

\begin{proposition}
\label{p:Golod-I}
Let $\varkappa\colon (P,\p,\kk)\to (R,\m,\kk)$ be a surjective  homomorphism of standard graded $\kk$-algebras. Assume there exists a positive integer $c$ such that, if \[\Tor^P_i(\pi, \kk)\colon\, \Tor^P_i(R/\m^{2c},\kk)\to \Tor^P_i(R/\m^{c}, \kk)\]
denotes the map induced by the canonical projection $\pi\colon R/\m^{2c}\to R/\m^{c}$, then 
\begin{equation}
\label{e:Tors}
\Tor^P_i(\pi, \kk)=0 \quad \text{for all $i\ge 1$}. 
\end{equation}

Then, $\varkappa$ is a Golod homomorphism under any of the following two assumptions. 
\begin{enumerate}[$(I)$]
\item There exists a positive integer $b$ such that for all $i\geq 1$ we have
\begin{equation}
\label{0}
\beta^P_{i,j}(R)=0 \quad \text{if}\quad  j-i<b\quad \text{or}\quad  j-i\ge 2b. 
\end{equation}
\item $\ch(R)\ne 2$ and there exist positive integers $b,i_0,j_0$  with $b\ge j_0-i_0\ge 1$ and $i_0$ odd such that for $i\geq 1$ we have
\begin{align}
 \beta_{i,j}^P(R) &= 0 \quad \text{if\,  } j-i<b \text{ \,and\, } (i,j)\ne (i_0,j_0)  \label{1}\\
 \beta_{i,j}^P(R)&=1  \quad \text{if\,  } (i,j)=(i_0,j_0) \label{2}\\
  \beta_{i,j}^P(R) &= 0 \quad \text{if\,  } j-i\geq b+j_0-i_0,\label{3}
\end{align}
\end{enumerate}
\end{proposition}

\begin{proof} 
Let $\mathcal A$ be as in \ref{Golod-hom}. 
First, remark that \eqref{e:Tors} is a reformulation of \eqref{map0}, as $\HH_i(\mathcal A/\m^u\mathcal A)=\Tor_i^P(R/\m^u,\kk)$ for all $u\ge 0$ and $i\ge 0$.
As in the proof of \Cref{l:Massey-lemma}, choose $\bsh=\{h_\lambda\}_{\lambda\in \Lambda}$ a homogeneous basis of the bigraded $\kk$-vector space $\HH_{\ge 1}(\cA)_*$ and homogeneous cycle $z_\lambda\in \m^c\mathcal A_{i_\lambda, j_\lambda}$ with $h_\lambda=[z_\lambda]$.  In particular, we have
\begin{equation}
\label{Hn0}
  H_{i_\lambda}(\mathcal A)_{j_\lambda}\not=0,\quad\text{for all }\lambda\in\Lambda.  
\end{equation}

The desired conclusion follows if we prove that for each $h\in\bsh^n$ with $n\geq 2$ one of the  conditions (1) and (2) of  \Cref{l:Massey-lemma} hold.

Assume (I) holds. Using that  $\HH_i(\mathcal A)_j=0$ for all $i,j$ with $j-i<b$, \eqref{Hn0} gives  $b\le j_\lambda-i_\lambda.$ Hence, the equalities in \eqref{e:ijh} yield
\[
j(h)-(i(h)-1)=\sum_{k=1}^n(j_{\lambda_k}-i_{\lambda_k})-(n-2)\ge bn-(n-2)=2b+(b-1)(n-2)\ge 2b,
\]
as $n\geq 2$. Moreover,  by \eqref{0} we also have  $\HH_i(\mathcal A)_j=0$ for all $i,j$ with $j-i\geq 2b$, so $\HH_{i(h)-1}(\mathcal A)_{j(h)}=0$.  Therefore, in this case,  the condition (1) of \Cref{l:Massey-lemma} holds. 

Now assume (II) holds. Conditions \eqref{Hn0} and \eqref{1} now give that 
\begin{equation}
\label{e:2-or}
(i_{\lambda},j_{\lambda})=(i_0,j_0)\quad \text{ or }\quad j_{\lambda}-i_{\lambda}\ge b\ge j_0-i_0\ge 1.
\end{equation}
If $h$ satisfies the condition (2) of \Cref{l:Massey-lemma}, then we are done. Next, if we assume that (2) does not hold, then we prove that $h$ must  satisfy condition (1) of \Cref{l:Massey-lemma}, that is $\HH_{i(h)-1}(\mathcal A)_{j(h)}=0$. 

 Since $\dim_\kk\HH_{i_0}(\mathcal A)_{j_0}=\beta^P_{i_0,j_0}(R)=1$ by \eqref{2},  there exists $\lambda_0\in \Lambda$ such that 
\begin{equation}
\label{e:lambda-i}
(i_{\lambda}, j_{\lambda})=(i_0,j_0)\iff \lambda=\lambda_0. 
\end{equation}

If $\lambda_k=\lambda_0$ for all $1\leq k\leq n$, then $z_{\lambda_i}z_{\lambda_j}=z_{\lambda_0}^2=0$ for all $1\leq i,j\leq n$, due to the graded-commutativity of the DG-algebra $\mathcal A$, since the homological degree $i_0$ of $z_{\lambda_0}$ is odd and $\ch(R)\ne 2$. This contradicts the assumption that condition (2) does not hold. 
Consequently, there exists $1\leq u\leq n$ such that $\lambda_u\ne \lambda_0$, and hence $(i_{\lambda_u}, j_{\lambda_u})\ne (i_0,j_0)$ by \eqref{e:lambda-i}. Then, we have the following (in)equalities:
\begin{align}
\begin{split}
\label{e:2-ineq}
j(h)-(i(h)-1)&=1+\sum_{k\in [n]\smallsetminus \{u\}} (j_{\lambda_k}-i_{\lambda_k})+(j_{\lambda_u}-i_{\lambda_u})-(n-1)\\
&\ge 1+(n-1)(j_0-i_0)+b-(n-1)\\
&= 1+(n-1)(j_0-i_0-1)+b\\
&\ge b+j_0-i_0
\end{split}
\end{align}
where the first equality follows from \eqref{e:ijh}, the first inequality follows from \eqref{e:2-or}, and the last inequality comes from $n\ge 2$. In view of \eqref{3}, it follows that $\HH_{i(h)-1}(\mathcal A)_{j(h)}=0$. We conclude that condition (1) of \cref{l:Massey-lemma} holds. 
 \end{proof}

The next result shows how the condition \eqref{e:Tors} in \Cref{p:Golod-I} can be verified in the setting of this paper. 

\begin{lemma}
\label{l:powers-zero-map}  Let $\kk$ be a field,  $P=\kk[x_1, \dots, x_n]/(f_1, \dots, f_u)$ where $\{f_1, \dots, f_u\}$ is a regular sequence of homogeneous elements of degree $a\ge 2$, and  $\varkappa\colon (P,\p,\kk)\to (R,\m,\kk)$  a surjective  homomorphism of standard graded rings such that $\Ker\varkappa\subseteq\p^a$.  Then,
the induced map
\[
\Tor^P_i(R/\m^{a}, \kk)\to \Tor^P_i(R/\m^{a-1},\kk)\quad\text{is zero for all $i>0$}. 
\]
 
\end{lemma}

\begin{proof} Since  $R\cong P/\Ker\varkappa$, we have $R/\m^a\cong P/(\Ker\varkappa+\p^a)\cong P/\p^a$. Similarly, we have $R/\m^{a-1}\cong P/\p^{a-1}$, as $\Ker\varkappa\subseteq\p^a\subseteq\p^{a-1}$. The desired conclusion now follows from the fact that the map induced by the inclusion $\p^a\subseteq \p^{a-1}$ 
\[
\Tor^P_i(\p^a,\kk)\to \Tor^P_i(\p^{a-1},\kk)\quad\text{is zero for all $i\geq 0$},
\]
by \cite[Corollary 2.6]{HS}.
\end{proof}

\section*{Acknowledgments} 
The project started at the workshop  “Women in Commutative Algebra II” (WICA II) held at CIRM Trento, Italy, October 16–20, 2023. The authors would like to thank the CIRM and the organizers for the invitation and financial support through the NSF grant DMS–2324929.
Support for attending the workshop also came from the AWM Travel Grants Program funded by NSF Grant DMS-2015440, Clay Mathematics Institute, and COS of Northeastern University.

This material was also partly  supported by the National Science Foundation under Grant No.\@ DMS-1928930 and by the Alfred P. Sloan Foundation under grant G-2021-16778, while the fifth author was in residence at the Simons Laufer Mathematical Sciences Institute (formerly MSRI) in Berkeley, California, during the Spring 2024 semester.  The first author was partially supported by the National Science Foundation Award No. 2418637. 

\bibliographystyle{abbrv}
\bibliography{references}

\end{document}